\documentclass[a4paper, 11pt]{amsart}
\usepackage[utf8]{inputenc}
\usepackage[english]{babel}
\usepackage[T1]{fontenc}
\usepackage{amsmath, amssymb, amsfonts}
\usepackage[all]{xy}
\usepackage{enumerate}
\usepackage{graphicx}
\usepackage{geometry}
\usepackage{tikz}
\usetikzlibrary{arrows}
\usepackage{hyperref}
\usepackage{comment}

\title{The KK-theory of Amalgamated Free Products}
\author{Pierre Fima and Emmanuel Germain}
\thanks{P.F. is partially supported by ANR grants OSQPI and NEUMANN.  E.G thanks CMI, Chennai for its support when part of this research was underway.}

\theoremstyle{plain}
\newtheorem{theorem}{Theorem}[section]
\newtheorem{proposition}[theorem]{Proposition}
\newtheorem{corollary}[theorem]{Corollary}
\newtheorem{lemma}[theorem]{Lemma}

\theoremstyle{definition}
\newtheorem{definition}[theorem]{Definition}
\newtheorem{example}[theorem]{Example}

\theoremstyle{remark}
\newtheorem{remark}[theorem]{Remark}

\DeclareMathOperator{\E}{E}

\newcommand{\EE}{\mathbb{E}}

\newcommand{\K}{\mathcal{K}}

\newcommand{\R}{\mathbb{R}}

\newcommand{\ot}{\otimes}
\newcommand{\id}{\text{id}}

\begin{document}

\begin{abstract}
In the presence of conditional expectations, we prove a long exact sequence in KK-theory for both the maximal and the vertex reduced amalgamated free product of unital C*-algebras that is valid even for non GNS-faithful conditional expectations. However, in the degenerated case, one has to introduce a new reduced amalgamated free product, that we call vertex-reduced. In the course of the proof we established the KK-equivalence between the full amalgamated free product and the vertex-reduced amalgamated free product. This results generalize and simplify the results obtained before by Germain and Thomsen. When the conditional expectations are extremely degenerated, i.e. when they are $*$-homomorphisms, our vertex-reduced amalgamated free product is isomorphic to the fiber direct sum. Hence our results also generalize a result of Cuntz.
\end{abstract}

\maketitle

\section{Introduction}

\noindent In 1982 J. Cuntz obtained a very elegant result about the full free product of unital C*-algebras with one-dimensional representations that leads to a conjectural long exact sequence for amalgamated free products in a general situation \cite{Cu82}. At about the same time M. Pimsner and D. Voiculescu computation of the $KK$-theory for some groups C*-algebras culminated in the computation of full and reduced crossed products by groups acting on trees \cite{Pi86} (or by the fundamental group of a graph of groups in Serre's terminology).  To go over the group situation has been difficult  and it relied heavily  on various generalizations of Voiculescu absorption theorem (see \cite{Th03} for the most general results in that direction).  Note also that G. Kasparov and G. Skandalis had another proof of Pimsner long exact sequence when studying KK-theory for buildings \cite{KS91}

\vspace{0.2cm}

\noindent  Section $2$ is a preliminary section in which we investigate the notion of reduced amalgamated free products of unital C*-algebras $A_1*_BA_2$ in the presence of non-necessary GNS-faithful conditional expectations. The usual reduced version, due to D. Voiculecscu, which is obtained by looking at the module over $B$, is often too small. Indeed, when the conditional expectations onto $B$ are both $*$-homomorphisms, the Voiculescu's reduced amalgamated free product is isomorphic to $B$ and all the information on $A_1$ and $A_2$ is lost. This is why we consider another reduced amalgamated free product, that we call vertex-reduced, which is obtained by looking at the two modules over $A_1$ and $A_2$ and is an intermediate quotient between the full amalgamated free product and the Voiculescu's reduced amalgamated free product. When the conditional expectations are GNS faithful, this two reduced amalgamated free products coincide and when the conditional expectations are $*$-homomorphisms the vertex reduced amalgamated free product is isomorphic to the fiber sum $A_1\oplus_B A_2$. Hence, even in the extreme degenerated case, the information on $A_1$ and $A_2$ is still contained in the vertex-reduced amalgamated free product. As the vertex-reduced free product is a new construction, we devote some time to show some of its properties.

\vspace{0.2cm}

\noindent Before proving our long exact sequence in KK-theory we start with an auxiliary and easy result in Section $3$. This result states that the full free product is always K-equivalent to  the vertex-reduced free product.
In particular, when the conditional expectations are morphisms, we get exactly Cuntz result \cite{Cu82}. This result also generalizes and simplifies the previous result obtained by the second author \cite{Ge96}. The proof is very natural, just a rotation trick. While finishing writing this paper, the authors have been aware that K. Hasegawa just obtained the same result in the particular case of GNS-faithful conditional expectations. By a remark by Ueda (\cite{Ue08}), this result also proves the K-equivalence between full and (vertex) reduced HNN extensions.

\vspace{0.2cm}

\noindent The main part and also the more difficult part of our paper comes in Section $4$. Under the presence of conditional expectation, we show that the full amalgamated free product $A_1\underset{B}{*} A_2$ with the algebra $D$ of continuous functions $f$ from $]-1,1[$ to the full free product such that $f(]-1,0])\subset A_1$, $f([0,1[)\subset A_2$ and $f(0)\in B$.  This is done by generalizing one of the author paper (\cite{Ge97}). Therefore the 
full amalgamated free product $A_1\underset{B}{*} A_2$ sits inside a long exact sequence for the computation of its $KK$-groups. Of course the vertex reduced free product has the same long exact sequence. Explicitly, if $C$ is any separable C*-algebra, then we have the two $6$-terms exact sequence,
$$\begin{array}{ccccc}
    KK^0(C,B)&\longrightarrow& KK^0(C,A_1)\bigoplus KK^0(C,A_2) &\longrightarrow &KK^0(C,A_1*_B A_2)\\
\uparrow& & & &\downarrow\\
KK^1(C,A_1*_BA_2)& \longleftarrow& KK^1(C,A_1)\bigoplus KK^1(C,A_2) &\longleftarrow& KK^1(C,B) \\
\end{array}
$$
and
$$\begin{array}{ccccc}
   KK^0(B,C)&\longleftarrow& KK^0(A_1,C)\bigoplus KK^0(A_2,C) &\longleftarrow& KK^0(A_1*_BA_2,C)\\
\downarrow& & & &\uparrow\\
KK^1(A_1*_BA_2,C)& \longrightarrow& KK^1(A_1,C)\oplus KK^1(A_2,C) &\longrightarrow&  KK^1(B,C) \\
\end{array}
$$

\noindent Again the HNN extension case follows using the isomorphism with an amalgamated free product. Note that this result greatly simplifies and generalizes the results of Thomsen \cite{Th03} about KK-theory for amalgamated free products which are valid only when the amalgam is finite dimensional.

\vspace{0.2cm}

\noindent Let us mention some applications. As a direct corollary, we obtain that the amalgamated free product of discrete quantum groups is $K$-amenable if and only if the initial quantum groups are K-amenable. This generalizes the result of Vergnioux \cite{Ve04} which was valid only for amenable discrete quantum groups and this also implies that a graph product of discrete quantum groups (see \cite{CF14}) is $K$-amenable if and only if the initial quantum groups are $K$-amenable. Finally, let us mention that this results will be applied in a future paper to deduce a long exact sequence in KK-theory for fundamental C*-algebras of graph of C*-algebras, generalizing and simplifying the results of Pimsner \cite{Pi86} and, as an application, the results of Fima-Freslon \cite{FF13}.

\section{Preliminaries}

\subsection{Notations and conventions} All C*-algebras and Hilbert modules are supposed to be separable. For a C*-algebra $A$ and a Hilbert $A$-module $H$ we denote by $\mathcal{L}_A(H)$ the C*-algebra of $A$-linear adjointable operators from $H$ to $H$ and by $\mathcal{K}_A(H)$ the sub-C*-algebra of $\mathcal{L}_A(H)$ consisting of $A$-compact operators. For $a\in A$, we denote by $L_A(a)\in\mathcal{L}_A(A)$ the left multiplication operator by $a$.

\subsection{Conditional expectations}\label{SectionCE}

Let $A$, $B$ be unital C*-algebras and $\varphi\,:\,A\rightarrow B$ be a unital completely positive map (ucp). A \textit{GNS construction} of $\varphi$ is a triple $(K,\rho,\eta)$, where $K$ is a Hilbert $B$-module, $\eta\in K$ and $\rho\,:\, A\rightarrow\mathcal{L}_B(K)$ is a unital $*$-homomorphism such that $K=\overline{\rho(A)\eta\cdot B}$ and $\langle\eta,\rho(a)\eta\rangle=\varphi(a)$ for all $a\in A$. A GNS construction always exists and is unique, up to a canonical isomorphism. Note that, if $B\subset A$ and $E\,:\,A\rightarrow B$ is a conditional expectation then, the Hilbert $B$-submodule $\eta\cdot B$ of $K$, where $(K,\rho,\eta)$ is a GNS construction of $E$, is complemented. Indeed, we have $K=\eta\cdot B\oplus K^\circ$, where $ K^\circ=\overline{\text{Span}}\{\rho(a)\eta\cdot b\,:\,a\in A^\circ\text{ and }b\in B\}$ and $A^\circ=\text{Ker}(E)$. Since $E$ is a conditional expectation onto $B$ we have $bA^\circ\subset A^\circ$ for all $b\in B$. It follows that $\rho(b)K^\circ\subset K^\circ$ for all $b\in B$. Hence, the restriction of $\rho$ to $B$ (and to $K^\circ$) gives a unital $*$-homomorphism $\rho\,:\,B\rightarrow\mathcal{L}_B(K^\circ)$.

\vspace{0.2cm}

\noindent A conditional expectation is called \textit{GNS-faithful} (or \textit{non-degenerate}) if for a given GNS construction (and hence for all GNS constructions) $(K,\rho,\eta)$, the homomorphism $\rho$ is faithful. In this paper we will consider reduced amalgamated free product with respect to non-necessary GNS-faithful conditional expectations. Actually, the degeneracy of the conditional expectations will naturally produce different types of reduced amalgamated free products. This is why we include the next proposition, which is well known to specialists but helps to understand the extreme degenerated case: when $E$ is an homomorphism. We include a complete proof for the convenience of the reader.

\begin{proposition}\label{PropDeg1} Let $B\subset A$ be a unital inclusion of unital C*-algebras and $E\,:\,A\rightarrow B$ be a conditional expectation with GNS construction $(K,\rho,\eta)$. The following are equivalent.
\begin{enumerate}
\item $E$ is an homomorphism.
\item $K\simeq B$ as Hilbert $B$-modules.
\item $K^\circ=\{0\}$.
\end{enumerate}
\end{proposition}
 
\begin{proof}
Since $K=\eta\cdot B\oplus K^\circ$ the equivalence between $(2)$ and $(3)$ is obvious.
\vspace{0.2cm}

\noindent$(1)\Rightarrow(3)$. If $E$ is an homomorphism from $A$ to $B$ then, since $E$ is ucp, it is a unital $*$-homomorphism and we have for all $b\in B$ and all $a\in A^\circ$,
$$\langle\rho(a)\eta\cdot b,\rho(a)\eta\cdot b\rangle_{K}=b^*\langle\eta\cdot b,\rho(a^*a)\eta\rangle_{K}b=b^*E(a^*a)b=b^*E(a)^*E(a)b=0.$$

\noindent$(3)\Rightarrow(1)$. If $K^\circ=\{0\}$ then, for all $a\in A^\circ$, we have $E(a^*a)=\langle\rho(a)\eta,\rho(a)\eta\rangle_{K}=0$. Hence
$$E((a-E(a))^*(a-E(a)))=0=E(a^*a)-E(a^*)E(a)-E(a)^*E(a)+E(a)^*E(a)\quad\text{for all }a\in A.$$
It follows that, for all $a\in A$, we have $E(a^*a)=E(a)^*E(a)$. Hence, the multiplicative domain of the ucp map $E$ is equal to $A$ which implies that $E$ is an homomorphism.
\end{proof}

\subsection{The full and reduced amalgamated free products}

\noindent Let $A_1$, $A_2$ be two unital C*-algebras with a common C*-subalgebra $B\subset A_k$, $k=1,2$ and denote by $A_f$ the full amalgamated free product. To be more precise, we sometimes write $A_f=A_1\underset{B}{*}A_2$. It is well known that the canonical map from $A_k$ to $A_f$ is faithful for $k=1,2$. Hence, we will always view $A_1$ and $A_2$ as subalgebras of $A_f$.

\vspace{0.2cm}

\noindent We will now construct, in the presence of conditional expectations, two different reduced amalgamated free products. One of them, that we call the \textit{edge-reduced amalgamated free product} has been extensively studied and it is called, in the literature, the reduced amalgamated free product. The other one, that we call the \textit{vertex-reduced amalgamated free product}, does not seem to be known, even from specialists. As it will become gradually clear, the vertex-reduced amalgamated free product is actually much more natural than the edge-reduced amalgamated free product. It is an intermediate quotient of the full amalgamated free product and it is isomorphic to the edge-reduced amalgamated free product in the presence of GNS-faithful conditional expectations. This is the reason why it has not appear before in the literature since many authors only consider amalgamated free product in the presence of GNS-faithful conditional expectations. Since the vertex-reduced and the edge-reduced amalgamated free product are the foundations of our proofs we will now explain in great details their constructions.
 
 \vspace{0.2cm}
 
\noindent In the sequel, we always assume that, for $k=1,2$, there exists a conditional expectation $E_k\,:\, A_k\rightarrow B$. We write $A_k^\circ=\{a\in A_k\,:\, E_k(a)=0\}$, we denote by $(K_k,\rho_k,\eta_k)$ a GNS construction of $E_k$ and by $K_k^\circ$ the canonical orthogonal complement of $\eta_k\cdot B$ in $K_k$ as explain in section \ref{SectionCE}. Recall that the restriction of $\rho_k$ to $B$ (and to $K_k^\circ$) gives a unital $*$-homomorphism $\rho_k\,:\, B\rightarrow\mathcal{L}_B(K_k^\circ)$.

\vspace{0.2cm}

\noindent We denote by $I$ the subset of $\cup_{n\geq 1}\{1,2\}^n$ defined by
$$I=\{(i_1,\dots,i_n)\in\{1,2\}^n\,:\,n\geq 1\text{ and }i_k\neq i_{k+1}\text{ for all }1\leq k\leq n-1\},$$

\noindent Recall that an operator $x\in A_f$ is called \textit{reduced} if $x\neq 0$ and $x$ can be written as $x=a_1\dots a_n$ with $n\geq 1$ and $a_k\in A_{i_k}^\circ-\{0\}$ such that $\underline{i}=(i_1,\dots i_n)\in I$.
\subsubsection{The vertex-reduced amalgamated free products} For $\underline{i}=(i_1,\dots,i_n)\in I$, we define a $A_{i_1}$-$A_{i_n}$-bimodule $H_{\underline{i}}$. As Hilbert $A_{i_n}$-module we have:

$$H_{\underline{i}}=\left\{\begin{array}{lcl}K_{i_1}\underset{B}{\ot}K_{i_2}^\circ\underset{B}{\ot}\dots\underset{B}{\ot}K_{i_{n-1}}^\circ\underset{B}{\ot} A_{i_n}&\text{if}&n\geq 3,\\
K_{i_1}\underset{B}{\ot} A_{i_2}&\text{if}&n=2,\\
A_{i_1}&\text{if}&n=1.\end{array}\right.$$

\noindent The left action of $A_{i_1}$ on $H_{\underline{i}}$ is given by the unital $*$-homomorphism defined by
$$\lambda_{\underline{i}}\,:\,A_{i_1}\rightarrow\mathcal{L}_{A_{i_n}}(H_{\underline{i}});\quad\lambda_{\underline{i}}=\left\{\begin{array}{lcl}\rho_{i_1}\underset{B}{\ot}\id&\text{if}& n\geq 2,\\ L_{A_{i_1}}&\text{if}&n=1.\end{array}\right.$$
\noindent  We consider, for $k,l\in\{1,2\}$, the subset $I_{k,l}=\{\underline{i}=(i_1,\dots,i_n)\in I\,:\,i_1=k\text{ and }i_n=l\}$ and the $A_k$-$A_l$-bimodule defined by
$$H_{k,l}=\underset{\underline{i}\in I_{k,l}}{\bigoplus}H_{\underline{i}}\quad\text{and}\quad\lambda_{k,l}=\underset{\underline{i}\in I_{k,l}}{\bigoplus}\lambda_{\underline{i}}\,:\, A_k\rightarrow\mathcal{L}_{A_l}(H_{k,l}).$$

\noindent For $k\in\{1,2\}$ we denote by $\overline{k}$ the unique element in $\{1,2\}\setminus\{k\}$.

\begin{example}\label{ExDeg1}
If, for $k\in\{1,2\}$, $E_k$ is an homomorphism from $A_k$ to $B$ it follows from Proposition \ref{PropDeg1} that $K_k^\circ=\{0\}$. Hence, $H_{k,k}=A_k\oplus K_k\underset{B}{\ot} K_{\overline{k}}^\circ\underset{B}{\ot} A_k$ and $H_{\overline{k},k}=K_{\overline{k}}\underset{B}{\ot} A_k$. Note that, since $K_k\simeq B$, we have $H_{k,k}\simeq A_k\oplus K_{\overline{k}}^\circ\underset{B}{\ot} A_k\simeq K_{\overline{k}}\underset{B}{\ot} A_k= H_{\overline{k},k}$. Also we have $H_{k,\overline{k}}=K_k\underset{B}{\ot} A_{\overline{k}}$ and $H_{\overline{k},\overline{k}}=A_{\overline{k}}$. Again, $H_{k,\overline{k}}\simeq A_{\overline{k}}=H_{\overline{k},\overline{k}}$. Actually the isomorphism of Hilbert $A_l$-modules $H_{k,l}\simeq H_{\overline{k},l}$ is true in full generality as explained below.
\end{example}

\vspace{0.2cm}

\noindent For $k,l\in\{1,2\}$ we define a unitary $u_{k,l}\in\mathcal{L}_{A_l}(H_{k,l}, H_{\overline{k},l})$, by the following formula. Let $\underline{i}=(i_1,\dots,i_n)\in I$, with $i_1=k$ and $i_l=l$. For $\xi\in H_{\underline{i}}$ we define $u_{k,l}\xi\in H_{\overline{k},l}$ in the following way.
\begin{itemize}
\item If $n\geq 2$, write $\underline{i}=(k,\underline{i}')$, where $\underline{i}'=(i_2,\dots,i_n)\in I_{\overline{k},l}$. For $\xi=\rho_{k}(a)\eta_k\ot\xi'$, with $a\in A_k$ and $\xi'\in H_{\underline{i}'}$, we define $u_{k,l}\xi:=\left\{\begin{array}{lcl}\eta_{\overline{k}}\ot\xi&\text{if}&E_k(a)=0,\\
\lambda_{\underline{i}'}(a)\xi'&\text{if}&a\in B.\end{array}\right.$
\item If $n=1$ then $k=l$, $\underline{i}=(l)$ and $\xi\in A_l=H_{\underline{i}}$. We define $u_{k,l}\xi:=\eta_{\overline{k}}\ot \xi$.
\end{itemize}
It is easy to check that, for all $k,l\in\{1,2\}$, the operator $u_{k,l}$ commutes with the right actions of $A_l$ on $H_{k,l}$ and $H_{\overline{k},l}$ and extends to a unitary operators, still denoted $u_{k,l}$, in $\mathcal{L}_{A_l}(H_{k,l}, H_{\overline{k},l})$ such that $u_{k,l}^*=u_{\overline{k},l}$. Moreover, the definition of $u_{k,l}$ implies that,
\begin{equation}\label{EqVertexUnitary}u_{k,l}^*\lambda_{\overline{k},l}(b)u_{k,l}=\lambda_{k,l}(b)\quad\text{for all }b\in B.
\end{equation}

\begin{definition}
Let $k\in\{1,2\}$. The \textit{$k$-vertex-reduced amalgamated free product} is the C*-sub-algebra $A_{v,k}\subset\mathcal{L}_{A_k}(H_{k,k})$ generated by $\lambda_{k,k}(A_k)\cup u_{k,k}^*\lambda_{\overline{k},k}(A_{\overline{k}})u_{k,k}\subset\mathcal{L}_{A_k}(H_{k,k})$. To be more precise, we use sometimes the notation $A_{v,k}=A_1\overset{k}{\underset{B}{*}}A_2$.
\end{definition}

\noindent For a fixed $k\in\{1,2\}$ the relations $(\ref{EqVertexUnitary})$ imply the existence of a unique unital $*$-homomorphism $\pi_k\,:\,A_f\rightarrow A_{v,k}$ such that $\pi_k(a)=\left\{\begin{array}{lcl}\lambda_{k,k}(a)&\text{if}&a\in A_k,\\
u_{k,k}^*\lambda_{\overline{k},k}(a) u_{k,k}&\text{if}&a\in A_{\overline{k}}.\end{array}\right.$

\vspace{0.2cm}

\noindent In the sequel we will denote by $\xi_k$ the vector $\xi_k:=1_{A_k}\in A_k\subset H_{k,k}$. We summarize the fundamental properties of $A_{v,k}$ in the following proposition.

\begin{proposition}\label{PropkVertexReduced}
For all $k\in\{1,2\}$ the following holds.
\begin{enumerate}
\item The morphism $\pi_k$ is faithful on $A_k$.
\item If $E_{\overline{k}}$ is GNS-faithful then $\pi_k$ is faithful on $A_{\overline{k}}$.
\item There exists a unique ucp map $\EE_k\,:\, A_{v,k}\rightarrow A_k$ such that $\EE_k(\pi_k(a))=a$ $\forall a\in A_k$ and
$$\EE_k(\pi_k(a_1\dots a_n))=0\text{ for all }a=a_1\dots a_n\in A_f\text{ reduced with }n\geq 2\text{ or }n=1\text{ and }a=a_1\in A_{\overline{k}}^\circ.$$
\noindent Moreover, $\EE_k$ is GNS-faithful.
\item For any unital C*-algebra $C$ with unital $*$-homomorphisms $\nu_k\,:\,A_k\rightarrow C$ such that
\begin{itemize}
\item $\nu_1(b)=\nu_2(b)$ for all $b\in B$,
\item $C$ is generated, as a C*-algebra, by $\nu_1(A_1)\cup\nu_2(A_2)$,
\item $\nu_k$ is faithful and there exists a GNS-faithful ucp map $E\,:\,C\rightarrow A_k$ such that $E(\nu_k(a))=a$ for all $a\in A_k$ and 
$$E(\nu_{i_1}(a_1)\dots\nu_{i_n}(a_n))=0\text{ for all }a=a_1\dots a_n\in A_f\text{ reduced with }n\geq 2\text{ or }n=1\text{ and }a=a_1\in A_{\overline{k}}^\circ,$$
\end{itemize}
there exists a unique unital $*$-isomorphism $\nu\,:\,A_{v,k}\rightarrow C$ such that $\nu\circ\pi_k(a)=\nu_k(a)$ for all $a\in A_1\cup A_2$. Moreover, $\nu$ satisfies $E\circ\nu=\EE_k$.
\end{enumerate}
\end{proposition}

\begin{proof} By definition of $\pi_k$ we have, if $a\in A_k$, $\langle\xi_k,\pi_k(a)\xi_k\rangle=a$. It follows directly that $\pi_k$ is faithful on $A_k$. Moreover, the map $\EE_k\,:\,A_{v,k}\rightarrow A_k$, $x\mapsto\langle\xi_k,x\xi_k\rangle$ satisfies $\EE_k(\pi_k(a))=a$ $\forall a\in A_k$. By the definition of the unitaries $u_{k,l}$ we have, for all $k\in\{1,2\}$ and all reduced operator $x=a_1\dots a_n$ with $a_k\in A_k^\circ$ and $\underline{i}=(i_1,\dots,i_n)\in I$,
\begin{equation}\label{EqGNS}\pi_k(a_1\dots a_n)\xi_k=\left\{\begin{array}{lcl}\rho_{i_1}(a_1)\eta_1\ot\dots\ot\rho_{i_{n-1}}(a_{n-1})\eta_{n-1}\ot a_n&\text{if}&i_1= k\text{ and }i_n=k,\\
\eta_{k}\ot\rho_{i_1}(a_1)\eta_1\ot\dots\ot\rho_{i_{n-1}}(a_{n-1})\eta_{n-1}\ot a_n&\text{if}&i_1\neq k\text{ and }i_n=k,\\
\rho_{i_1}(a_1)\eta_1\ot\dots\ot\rho_{i_{n}}(a_{n})\eta_{n}\ot 1_{A_k}&\text{if}&i_1= k\text{ and }i_n\neq k,\\
\eta_{k}\ot\rho_{i_1}(a_1)\eta_1\ot\dots\ot\rho_{i_{n}}(a_{n})\eta_{n}\ot 1_{A_k}&\text{if}&i_1\neq k\text{ and }i_n\neq k.\\
\end{array}\right.\end{equation}

\vspace{0.2cm}

\noindent Hence we have $\EE_k(\pi_k(a_1,\dots a_n))=0$ for all $a=a_1\dots a_n\in A_f$ reduced with $n\geq 2$ or $n=1$ and $a=a_1\in A_{\overline{k}}^\circ$. It also follows easily from the previous set of equations that $\overline{\pi_k(A_f)\xi_k\cdot A_k}=H_{k,k}$. Hence the triple $(H_{k,k},\id,\xi_k)$ is a GNS construction for $\EE_k$. This shows that $\EE_k$ is GNS-faithful. Note that the uniqueness statement of the third assertion is obvious since $A_f$ is the linear span of $B$ and the reduced operators. Also, the second statement becomes now obvious since, by the properties of $\EE_k$ we have, for all $x\in A_{\overline{k}}$, $\EE_k(\pi_k(x))=\EE_k(\pi_k(x-E_{\overline{k}}(x)))+\EE_k(\pi_k(E_{\overline{k}}(x)))=\pi_k(E_{\overline{k}}(x))$. It follows easily from this equation that $\pi_k$ is faithful on $A_{\overline{k}}$ whenever $E_{\overline{k}}$ is GNS-faithful. Indeed, let $x\in A_{\overline{k}}$ such that $\pi_k(a)=0$. Then, for all $y\in A_{\overline{k}}$ we have $\pi_k(y^*x^*xy)=0$. Hence, $\pi_k\circ E_{\overline{k}}(y^*x^*xy)=\EE_k\circ\pi_k(y^*x^*xy)=0$ for all $y\in A_{\overline{k}}$. Since $\pi_k$ is faithful on $A_k$ we find $E_{\overline{k}}(y^*x^*xy)=0$, for all $y\in A_{\overline{k}}$. Since $E_{\overline{k}}$ is GNS-faithful we conclude that $x=0$. 
\vspace{0.2cm}

\noindent $(4)$. The proof is a routine. We write the argument for the convenience of the reader. Let $(K,\rho,\eta)$ be the GNS construction of $E$. Since $E$ is GNS-faithful we may and will assume that $\rho=\id$ and $C\subset\mathcal{L}_{A_k}(K)$. By the properties of $\EE_k$ and $E$, the map $U\,:\,H_{k,k}\rightarrow K$ defined by, for $x=a_1\dots a_n\in A_f$ reduced with $a_k\in A_{i_k}^\circ$, $U(\pi_k(x)\xi_k):=\nu_{i_1}(a_1)\dots\nu_{i_n}(a_n)\eta$ and, for $x=b\in B$, $U(\pi_k(b)\xi_k)=\nu_1(b)\eta=\nu_2(b)\eta$, is well defined and extends to a unitary $U\in\mathcal{L}_{A_k}(H_{k,k},K)$. By construction, the map $\nu(x):=UxU^*$, for $x\in A_{v,k}$, satisfies the claimed properties. The uniqueness is obvious.
\end{proof}

\begin{remark}
It is known that the canonical homomorphism from $A_k$ to $A_f$ is faithful for $k\in\{1,2\}$ without assuming the existence of conditional expectations from $A_k$ to $B$. However, assertion $(1)$ of Proposition \ref{PropkVertexReduced} gives a very simple proof of this fact, since it shows that the composition of the canonical homomorphism from $A_k$ to $A_f$ with the homomorphism $\pi_k$ is faithful, which implies that the canonical homomorphism from $A_k$ to $A_f$ itself is faithful.
\end{remark}

\begin{example}\label{ExDeg2}
Suppose that, for a given $k\in\{1,2\}$, $E_k$ is an homomorphism. Then, as observed in Example \ref{ExDeg1}, we have
$H_{\overline{k},\overline{k}}=A_{\overline{k}}$ (and $\lambda_{\overline{k},\overline{k}}=L_{A_{\overline{k}}}$). It follows from the definition of $\pi_{\overline{k}}$ that
$$\pi_{\overline{k}}(a)=\left\{\begin{array}{lcl}L_{A_{\overline{k}}}(a)&\text{if}&a\in A_{\overline{k}},\\
0&\text{if}&a\in A_k^\circ.\end{array}\right.$$
Hence, since $A_f$ the closed linear span of $A_{\overline{k}}$ and the reduced operators and $\pi_{\overline{k}}\,:\,A_f\rightarrow A_{v,\overline{k}}$ is surjective, we find that $A_{v,\overline{k}}=\pi_{\overline{k}}(A_{\overline{k}})$. Moreover, since $\pi_{\overline{k}}$ is faithful on $A_{\overline{k}}$ we conclude that the restriction of $\pi_{\overline{k}}$ to $A_{\overline{k}}$ gives an isomorphism $A_{\overline{k}}\simeq A_{v,\overline{k}}$. 
\end{example}

\begin{definition}
The \textit{vertex-reduced amalgamated free product} is the C*-algebra obtained by separation and completion of $A_f$ with respect to the C*-semi-norm $\Vert\cdot\Vert_v$ on $A_f$ defined by
$$\Vert x\Vert_v:=\text{Max}\{\Vert\pi_1(x)\Vert,\Vert\pi_2(x)\Vert\}\quad\text{for all }x\in A_f.$$
\end{definition}

\noindent We will note it $A_1\overset{v}{\underset{B}{*}} A_2$ or $A_v$ for simplicity  in the rest of this section and let $\pi\,:\,A_f\rightarrow A_v$ be the canonical surjective unital $*$-homomorphism. Note that, by construction of $A_v$, for all $k\in\{1,2\}$, there exists a unique unital (surjective) $*$-homomorphism $\pi_{v,k}\,:\,A_v\rightarrow A_{v,k}$ such that $\pi_{v,k}\circ\pi=\pi_k$. We describe the fundamental properties of the vertex-reduced amalgamated free product in the following proposition. We call a family of ucp maps $\{\varphi_i\}_{i\in I}$, $\varphi_i\,:\,A\rightarrow B_i$ GNS-faithful if $\cap_{i\in I}{\rm Ker}(\pi_i)=\{0\}$, where $(H_i,\pi_i,\xi_i)$ is a GNS-construction for $\varphi_i$. From Proposition \ref{PropkVertexReduced} and the definition of $A_v$ we deduce the following result.

\begin{proposition}\label{PropVertexReduced}
The following holds.
\begin{enumerate}
\item $\pi$ is faithful on $A_k$ for all $k\in\{1,2\}$.
\item For all $k\in\{1,2\}$, there is a unique ucp map $\EE_{A_k}\,:\,A_v\rightarrow A_k$ such that $\EE_{A_k}\circ\pi(a)=a$ for all $a\in A_k$ and all $k\in\{1,2\}$ and,
$$\EE_{A_k}(\pi(a_1\dots a_n))=0\text{ for all }a=a_1\dots a_n\in A_f\text{ reduced with }n\geq 2\text{ or }n=1\text{ and }a=a_1\in A_{\overline{k}}^\circ.$$
\noindent Moreover, the family $\{\EE_{A_1},\EE_{A_2}\}$ is GNS-faithful.
\item Suppose that $C$ is a unital C*-algebra with $*$-homomorphisms $\nu_k\,:\,A_k\rightarrow C$ such that
\begin{itemize}
\item $\nu_1(b)=\nu_2(b)$ for all $b\in B$,
\item $C$ is generated, as a C*-algebra, by $\nu_1(A_1)\cup\nu_2(A_2)$,
\item $\nu_1$ and $\nu_2$ are faithful and, for all $k\in\{1,2\}$, there exists a ucp map $E_{A_k}\,:\,C\rightarrow A_k$ such that $E_{A_k}\circ\nu_k(a)=a$ for all $a\in A_k$ and all $k\in\{1,2\}$ and,
$$E_{A_k}(\nu_{i_1}(a_1)\dots\nu_{i_n}(a_n))=0\text{ for all }a=a_1\dots a_n\in A_f\text{ reduced with }n\geq 2\text{ or }n=1\text{ and }a=a_1\in A_{\overline{k}}^\circ,$$
and the family $\{E_{A_1},E_{A_2}\}$ is GNS-faithful.
\end{itemize}
Then, there exists a unique unital $*$-isomorphism $\nu\,:\,A_{v}\rightarrow C$ such that $\nu\circ\pi(a)=\nu_k(a)$ for all $a\in A_k$ and all $k\in\{1,2\}$. Moreover, $\nu$ satisfies $E_{A_k}\circ\nu=\EE_{A_k}$, $k\in\{1,2\}$.
\end{enumerate}
\end{proposition}

\begin{proof}
$(1)$. It is obvious since, by Proposition \ref{PropkVertexReduced}, $\pi_k$ is faithful on $A_k$ for $k=1,2$.
\vspace{0.2cm}

\noindent$(2)$. By Proposition \ref{PropkVertexReduced}, the maps $\EE_{A_k}=\EE_k\circ\pi_{v,k}$ satisfy the desired properties and it suffices to check that the family $\{\EE_{A_1},\EE_{A_2}\}$ is GNS-faithful. Let $x_0\in A_f$ be such that $x=\pi(x_0)\in A_v$ satisfies $\EE_{A_k}(y^*x^*xy)=0$ for all $y\in A_v$ and all $k\in\{1,2\}$. Then, for all $k\in\{1,2\}$ we have $\EE_k(y^*\pi_{v,k}(x^*x)y)=0$ for all $y\in A_{v,k}$. Since $\EE_k$ is GNS-faithful, this implies that $\pi_{v,k}(x)=\pi_k(x_0)=0$ for all $k\in\{1,2\}$. Hence, $\Vert x\Vert_{A_v}=\text{Max}(\Vert\pi_1(x_0)\Vert,\Vert\pi_2(x_0)\Vert)=0$.
\vspace{0.2cm}

\noindent$(3)$. The proof is a routine. We include it for the convenience of the reader. Let $(L_k,m_k,f_k)$ be the GNS construction of $E_{A_k}$. By the universal property of $A_{v,k}$, the C*-algebra $C_k$ generated by $m_k( C )\subset\mathcal{L}_{A_k}(L_k)$ is canonically isomorphic with $A_{v,k}$. Hence, in the remainder of the proof we suppose that $C_k=A_{v,k}$ and, by the universal property of $A_f$, we have a unital surjective $*$-homomorphism $\nu_f\,:\,A_f\rightarrow C$ such that $\nu_f\vert_{A_k}=\nu_k$. Note that, by the identification we made, $m_k\circ\nu_f=\pi_{k}$. Hence, by construction of $A_v$, there exists a unique unital (surjective) $*$-homomorphism $\nu\,:\,C\rightarrow A_v$ such that $\pi_{v,k}\circ\nu=m_k$ for all $k\in\{1,2\}$. The homomorphism $\nu$ satisfies all the claimed properties and it suffices to check that it is faithful. But it is obvious since, by the identity $\pi_{v,k}\circ\nu=m_k$, $k=1,2$, it follows that $\text{Ker}(\nu)\subset\text{Ker}(m_1)\cap\text{Ker}(m_2)=\{0\}$, since the pair $(E_{A_1},E_{A_2})$ is GNS-faithful.
\end{proof}

\begin{corollary}\label{CorDegVertexRed}
If both $E_1$ and $E_2$ are homomorphisms then there is a canonical isomorphism $A_v\simeq  A_1\underset{B}{\oplus}A_2$, where $A_1\underset{B}{\oplus}A_2:=\{(a_1,a_2)\in A_1\oplus A_2\,:\, E_1(a_1)=E_2(a_2)\}$.
\end{corollary}
\begin{proof}
We use the universal property of $A_v$ described in Proposition \ref{PropVertexReduced}. Define $\nu_k\,:\,A_k\rightarrow A_1\underset{B}{\oplus}A_2$ by $\nu_1(x)=(x,E_1(x))$ and $\nu_2(y)=(E_2(y),y)$. It is clear that $\nu_1$ and $\nu_2$ are both faithful unital $*$-homomorphisms such that $\nu_1(b)=\nu_2(b)$ for all $b\in B$. Define $E_{A_k}\,:\,A_1\underset{B}{\oplus}A_2\rightarrow A_k$ by $E_{A_1}(a_1,a_2)=a_1$ and $E_{A_2}(a_1,a_2)=a_2$. Then, for all $k\in\{1,2\}$, $E_k$ is a unital $*$-homomorphisms such that $E_{A_k}\circ\nu_k(a)=a$ for all $a\in A_k$. In particular both $E_1$ and $E_2$ are conditional expectations and, since ${\rm Ker}(E_{A_1})\cap{\rm Ker}(E_{A_2})=\{0\}$, the family $\{E_{A_1},E_{A_2}\}$ is GNS-faithful. Hence, it suffices to check the condition on the reduced operators. Since $\nu_1(A_1^\circ)=\{(x,0)\,:\,x\in A_1^\circ\}$ and $\nu_2(A_2^\circ)=\{(0,y)\,:\,y\in A_2^\circ\}$, we have $\nu_1(A_1^\circ)\nu_2(A_2^\circ)=\nu_2(A_2^\circ)\nu_1(A_1^\circ)=\{0\}$. Hence, it suffices to check the condition on elements $(a_1,a_2)\in\nu_1(A_1^\circ)\cup \nu_2(A_2^\circ)$ which is obvious.
\end{proof}
\subsubsection{The edge-reduced amalgamated free product}
In this section we show how the construction of the edge-reduced (or, in the literature, the reduced) amalgamated free product in full generality is related to
the vertex-reduced free product we just defined.
\vspace{0.2cm}

\noindent For $\underline{i}\in I$, we consider the $B$-$B$-module $K_{\underline{i}}=K_{i_1}^\circ\underset{B}{\ot}\dots\underset{B}{\ot}K_{i_n}^\circ$ as Hilbert $B$-module with the left action of $B$ given by the unital $*$-homomorphism $\rho_{\underline{i}}\,:\,B\rightarrow\mathcal{L}_B(K_{\underline{i}})$, $\rho_{\underline{i}}(b)=\rho_{i_1}(b)\underset{B}{\ot}\id$ for all $b\in B$ and we define the Hilbert $B$-bimodule $K=B\oplus\left(\bigoplus_{\underline{i}\in I}K_{\underline{i}}\right)$.

\begin{example}If, for some $k\in\{1,2\}$, $E_k$ is an homomorphism then $K=B\oplus K_{\overline{k}}^\circ\simeq K_{\overline{k}}$. Hence, if both $E_1$ and $E_2$ are homomorphisms then $K=B$.
\end{example}

\begin{proposition} There are isomorphisms  between $H_{k,k}\underset{E_k}{\ot} B$ and $K$ for $k=1,2$ implemented by some unitary $V_k$.
Moreover when we intertwine the representation $\pi_k\otimes 1$ by $V_k$ we get the classical representation of the full amalgamated free product on the space $K$.
\end{proposition}

\begin{proof}
Note that, for $\underline{i}=(i_1,\dots,i_n)\in I$ with $i_1=i_n=k$ (hence $n$ is odd) we have, if $n=1$, $H_{\underline{i}}\underset{E_k}{\ot} B=A_k\underset{E_k}{\ot} B\simeq K_k\simeq K_k^\circ\oplus B$,
and, if $n\geq 3$, $H_{\underline{i}}\underset{E_k}{\ot} B=K_k\underset{B}{\ot}\left(K_{\overline{k}}^\circ\underset{B}{\ot}\dots\underset{B}{\ot} K_{\overline{k}}^\circ\right)\underset{B}{\ot} K_k\simeq K_{\underline{i}}\oplus K_{\underline{i}'}\oplus K_{\underline{i}''}\oplus K_{\underline{i}'''}$, where $\underline{i}'=(i_2,\dots,i_n)$, $\underline{i}''=(i_1,\dots,i_{n-1})$ and $\underline{i}'''=(i_2,\dots,i_{n-1})$. Hence the existence of $V_k\,:\,H_{k,k}\underset{E_k}{\ot} B\rightarrow K$. It is easy to check that $V_k$ satisfies $V_k(\pi_k(a)\ot 1)V_k^*=\rho(a)$ for all $a\in A_k$ and all $k\in\{1,2\}$ where $\rho$ is the (classical) reduced free product representation which we recall below for convenience. For $l\in\{1,2\}$ define $K(l)=B\oplus\left(\underset{\underline{i}\in I,\,i_1\neq l}{\bigoplus} K_{\underline{i}}\right)$ and note that we have a unital $*$-homomorphism $\rho_{l}\,:\,B\rightarrow\mathcal{L}_{B}(K(l))$ defined by $\rho_{l}=\underset{\underline{i}\in I_k,\,i_1\neq l}{\bigoplus}\rho_{\underline{i}}$. Let $U_l\in\mathcal{L}_{B}(K_l\underset{\rho_{l}}{\ot} K(l),K)$ be the unitary operator defined by

$$\begin{array}{llcl}
U_l\,:\,&K_l\underset{\rho_{l}}{\ot}K(l) &\longrightarrow& K\\
&\eta_l\underset{\rho_l}{\ot} B &\overset{\simeq}{\longrightarrow}&B\\
&K_l^\circ\underset{\rho_l}{\ot} B&\overset{\simeq}{\longrightarrow}&K_l^\circ\\
&\eta_l\underset{\rho_l}{\ot} H_{\underline{i}}&\overset{\simeq}{\longrightarrow}&H_{\underline{i}}\\
&K_l^\circ\underset{\rho_l}{\ot} H_{\underline{i}}&\overset{\simeq}{\longrightarrow}&H_{(l,\underline{i})}\\
\end{array}$$
where $(l,\underline{i})=(l,i_1,\dots,i_n)\in I$ if $\underline{i}=(i_1,\dots,i_n)\in I$ with $i_1\neq l$. We define the unital $*$-homomorphisms $\lambda_{l}\,:\,\mathcal{L}_B(K_l)\rightarrow\mathcal{L}_{B}(K)$ by $\lambda_{l}(x)=U_l(x\ot1)U_l^*$. By definition we have $\lambda_{1}(\rho_1(b))=\lambda_{2}(\rho_2(b))$ for all $b\in B$. It follows that there exists a unique unital $*$-homomorphism $\rho\,:\,A_f\rightarrow \mathcal{L}_{B}(K)$ such that $\rho(a)=\lambda_k(a)$ for $a\in A_k$, for all $k\in\{1,2\}$.
\end{proof}

\begin{definition} The \textit{edge-reduced} amalgamated free product is the C*-subalgebra $A_e\subset\mathcal{L}_B(K)$ generated by $\lambda_1(A_1)\cup\lambda_2(A_2)\subset\mathcal{L}_B(K)$. To be more precise, we use sometimes the notation $A_e=A_1\overset{e}{\underset{B}{*}} A_2$.
\end{definition}

\begin{example}
If, for some $k\in\{1,2\}$, $E_k$ is an homomorphism then $A_e$ is the C*-algebra $\overline{\rho_{\overline{k}}(A_{\overline{k}})}\subset\mathcal{L}_B(K_{\overline{k}})$. If both $E_1$ and $E_2$ are homomorphisms then $A_e\simeq B$.
\end{example}

\noindent The preceding example shows that the edge reduced amalgamated free product may forget everything about the initial C*-algebras $A_1$ and $A_2$ in the extreme degenerated case: it only remembers $B$. This shows that, in general, one should consider instead the vertex-reduced amalgamated free product. Indeed, even in the extreme degenerated case, the vertex reduced amalgamated free product correctly remembers the C*-algebras $A_1$ and $A_2$, as shown in corollary \ref{CorDegVertexRed}.

\vspace{0.2cm}

\noindent  In the following proposition we recall the properties of $A_e$. The results below are well known when $E_1$ and $E_2$ are GNS-faithful. The proof is similar to the proof of proposition \ref{PropkVertexReduced} and we leave it to the reader.

\begin{proposition}\label{PropEdgeReduced}
The following holds.
\begin{enumerate}
\item $\rho$ is faithful on $B$.
\item If $\E_k$ is GNS-faithful then $\rho$ is faithful on $A_k$.
\item There exists a unique ucp map $\EE\,:\, A_{e}\rightarrow B$ such that $\EE\circ\rho(b)=b$ for all $b\in B$ and,
$$\EE(\rho(a_1,\dots a_n))=0\text{ for all }a=a_1\dots a_n\in A_f\text{ reduced}.$$
\noindent Moreover, $\EE$ is GNS-faithful.
\item For any unital C*-algebra $C$ with unital $*$-homomorphisms $\nu_k\,:\,A_k\rightarrow C$ such that
\begin{itemize}
\item $\nu_1(b)=\nu_2(b)$ for all $b\in B$,
\item $C$ is generated, as a C*-algebra, by $\nu_1(A_1)\cup\nu_2(A_2)$,
\item $\nu_1\vert_B=\nu_2\vert_B$ is faithful and there exists a GNS-faithful ucp map $E\,:\,C\rightarrow B$ such that $E\circ\nu_k(b)=b$ for all $b\in B$, $k=1,2$, and,
$$E(\nu_{i_1}(a_1)\dots\nu_{i_n}(a_n))=0\text{ for all }a=a_1\dots a_n\in A_f\text{ reduced},$$
\end{itemize}
there exists a unique unital $*$-isomorphism $\nu\,:\,A_{e}\rightarrow C$ such that $\nu\circ\rho(a)=\nu_k(a)$ for all $a\in A_k$, $k\in \{1,2\}$. Moreover, $\nu$ satisfies $E\circ\nu=\EE$.
\end{enumerate}
\end{proposition}

\begin{proposition}
For all $k\in\{1,2\}$ there exists a unique unital $*$-homomorphism
$$\lambda_{v,k}\,:\,A_{v,k}\rightarrow A_e\quad\text{such that}\quad\lambda_{v,k}\circ\pi_k=\rho.$$
Moreover, $\lambda_{v,k}$ is faithful on $\pi_k(A_{\overline{k}})$ and, if $E_k$ is GNS-faithful, $\lambda_{k,v}$ is an isomorphism.
\end{proposition}

\begin{proof}
The formulae $\lambda_{v,k}(x)=V_k(x\ot 1)V_k^*$ defines a unital $*$-homomorphism $\lambda_{v,k}\,:\,A_{v,k}\rightarrow A_e$  satisfying $\lambda_{v,k}\circ\pi_k=\rho$. The uniqueness of $\lambda_{v,k}$ is obvious. Let us check that $\lambda_{v,k}$ is faithful on $\pi_k(A_{\overline{k}})$. Suppose that $x\in A_{\overline{k}}$ and $\lambda_{v,k}(\pi_k(x))=0$. Then, for all $y\in A_{\overline{k}}$, we have $\rho(y^*x^*xy)=\lambda_{v,k}(\pi_k(y^*x^*xy))=0$. Hence, $0=\EE\circ\rho(y^*x^*xy)=\EE\circ\rho(E_{\overline{k}}(y^*x^*xy))=E_{\overline{k}}(y^*x^*xy)$. It follows that $x\in\text{Ker}(\rho_{\overline{k}})$ hence, $\lambda_{\overline{k},k}(x)=\oplus_{\underline{i}\in I_{\overline{k},k}}\rho_{\overline{k}}(x)\ot 1=0$ which implies that $\pi_k(x)=u_{k,k}^*\lambda_{\overline{k},k}(x)u_{k,k}=0$. The last statement follows from the universal property of $A_e$ since the ucp map $E_k\circ\EE_k\,:\,A_{v,k}\rightarrow B$ is GNS-faithful whenever $E_k$ is GNS-faithful.
\end{proof}

\noindent In the next proposition, we study some associativity properties between the edge-reduced and the vertex-reduced amalgamated free product. The result is interesting in itself and it will be used to easily obtain ucp radial multipliers on the vertex-reduced amalgamated free product.

\begin{proposition}\label{CorVertex-EdgeReduced}
Let $A_1,A_2,A_3$ be unital C*-algebras with a common unital C*-subalgebra $B$ and conditional expectations $\E_k\,:\,A_k\rightarrow B$. After identification of $A_1$ with a C*-subalgebra of both $A_1\overset{1}{\underset{B}{*}}A_2$ and $A_1\overset{1}{\underset{B}{*}}A_3$, the canonical GNS-faithful ucp maps $A_1\overset{1}{\underset{B}{*}}A_2\rightarrow A_1$ and $A_1\overset{1}{\underset{B}{*}}A_3\rightarrow A_1$ become conditional expectations and, with respect to this GNS-faithful conditional expectations, we have canonical isomorphisms
\begin{itemize}
\item $\left(A_1\overset{1}{\underset{B}{*}}A_2\right)\overset{e}{\underset{A_1}{*}}\left( A_1\overset{1}{\underset{B}{*}}A_3\right)\simeq A_1\overset{1}{\underset{B}{*}}\left(A_2\overset{e}{\underset{B}{*}}A_3\right)$.
\item $\left(A_1\overset{2}{\underset{B}{*}}A_2\right)\overset{e}{\underset{A_2}{*}}\left( A_3\overset{2}{\underset{B}{*}}A_2\right)\simeq \left(A_1\overset{e}{\underset{B}{*}}A_3\right)\overset{2}{\underset{B}{*}}A_2$.
\end{itemize}
\end{proposition}

\begin{proof}
We prove the first point. The proof of the second point is similar. We write $\widetilde{A}=A_1\overset{1}{\underset{B}{*}}\left(A_2\overset{e}{\underset{B}{*}}A_3\right)$. Let $\rho\,:\,A_2\underset{B}{*}A_3\rightarrow A_1\overset{e}{\underset{B}{*}}A_3$ and $\widetilde{\pi}\,:\,A_1\underset{B}{*}\left( A_2\overset{e}{\underset{B}{*}}A_3\right)\rightarrow \widetilde{A}$ be the canonical surjections and $\widetilde{\EE}\,:\,\widetilde{A}\rightarrow A_1$ the canonical GNS-faithful ucp map. Define, for $k=1,2$, $\nu_k\,:\,A_k\rightarrow D$ by $\nu_1=\widetilde{\pi}\vert_D$ and $\nu_2=\widetilde{\pi}\circ\rho\vert_{A_2}$. By definition, $\nu_1(b)=\nu_2(b)$ for all $b\in B$ and $\nu_1$ is faithful. Let $C$ be the C*-subalgebra of $\widetilde{A}$ generated by $\nu_1(A_1)\cup\nu_2(A_2)$. We claim that the exists a (unique) unital faithful $*$-homomorphism $\nu\,:\,A_1\underset{B}{\overset{1}{*}}A_2\rightarrow \widetilde{A}$ such that $\nu\circ\pi_1\vert_{A_k}=\nu_k$ for $k=1,2$, where $\pi_1\,:\,A_1\underset{B}{*}A_2\rightarrow A_1\underset{B}{\overset{1}{*}}A_2$ is the canonical surjection. By the universal property of the $1$-vertex-reduced amalgamated free product, it suffices to show the following claim, where $E=\widetilde{\EE}\vert_C\,:\,C\rightarrow A_1$.

\vspace{0.2cm}

\noindent\textbf{Claim.}\textit{ The ucp map $E$ is GNS-faithful and satisfies $E\circ\nu_1=\id_{A_1}$ and, for all $a=a_1\dots a_n\in A_f$ reduced with $a_k\in A_{i_k}^\circ$,
$E(\nu_{i_1}(a_1)\dots\nu_{i_n}(a_n))=0$ whenever $n\geq 2$ or $n=1$ and $a=a_1\in A_2^\circ$.}
\vspace{0.2cm}

\noindent\textit{Proof of the claim.} The fact the $E$ vanishes on the reduced operators (not in $A_1^\circ$) is obvious, since $\widetilde{\EE}$ satisfies the same property. The only non-trivial property to check is the fact that $E$ is GNS-faithful: indeed, it is not true, in general, that the restriction of a GNS-faithful ucp map to a subalgebra is again GNS-faithful. So suppose that there exists $x\in C$ such that $E(y^*x^*xy)=0$ for all $y\in C$ and let us show that $x$ must be zero. Since $\widetilde{\EE}\,:\,\widetilde{A}\rightarrow A_1$ is GNS-faithful, it suffices to show that $\widetilde{\EE}(y^*x^*xy)=0$ for all $y\in\widetilde{A}$. By hypothesis, we know that it is true for all $y\in C$. Since $\widetilde{A}$ is the closed linear span of $\widetilde{\pi}(A_1)$ and $\widetilde{\pi}(z)$, for $z\in A_1\underset{B}{*}\left( A_2\overset{e}{\underset{B}{*}}A_3\right)$ a reduced operator not in $A_1^\circ$ and since $\widetilde{\pi}(A_1)\cup\widetilde{\pi}\circ\rho(A_2)\subset C$, it suffices to show that $\widetilde{\EE}(y^*x^*xy)=0$ for $y=\widetilde{\pi}(z)$ and  $z=z_1\dots z_n\in A_1\underset{B}{*}\left( A_1\overset{e}{\underset{B}{*}}A_3\right)$ a reduced operator with letters $z_k$ alternating from $A_1^\circ$, $\rho(A_2^\circ)$ and $\rho(A_3^\circ)$ and containing at least one letter in $\rho(A_3^\circ)$. Since one of the $z_k$ is in  $\rho(A_3^\circ)$ and $x\in C$ we have, by the property of $\widetilde{\EE}$, $\widetilde{\EE}(y^*(x^*x-\widetilde{\EE}(x^*x))y)=0$. Hence, $\widetilde{\EE}(y^*x^*xy)=\widetilde{\EE}(y^*\widetilde{\EE}(x^*x)y)=\widetilde{\EE}(y^*E(x^*x)y)=0$, since $E(x^*x)=0$.

\vspace{0.2cm}

\noindent\textit{End of the proof of the proposition.} Define, for $k=1,3$, the unital $*$-homomorphism $\eta_k\,:\,A_k\rightarrow \widetilde{A}$ by $\eta_1=\widetilde{\pi}\vert_{A_1}=\nu_1$ and $\eta_3=\widetilde{\pi}\circ\rho\vert_{A_3}$. Using the universal property of the $1$-vertex-reduced amalgamated free product one can show, using exactly the same arguments we used to construct the homomorphism $\nu$, that there exists a (necessarily unique) unital faithful $*$-homomorphism $\eta\,:\, A_1\underset{B}{\overset{1}{*}}A_3\rightarrow\widetilde{A}$ such that $\eta\circ\pi_1'\vert_{A_k}=\eta_k$ for $k=1,3$, where $\pi_1'\,:\,A_1\underset{B}{*}A_3\rightarrow A_1\underset{B}{\overset{1}{*}}A_3$ is the canonical surjection. Note that $\nu(b)=\eta(b)$ for all $b\in B$ and $\widetilde{A}$ is generated, as a C*-algebra, by $\nu(A_1\underset{B}{\overset{1}{*}}A_2)\cup\eta(A_1\underset{B}{\overset{1}{*}}A_3)$. Since the GNS-faithful ucp map $\widetilde{\EE}\,:\,\widetilde{A}\rightarrow A_1$ obviously satisfies the condition on the reduced operators we may use the universal property of the edge-reduced amalgamated free product to conclude that there exists a canonical $*$-isomorphism
$$\left(A_1\overset{1}{\underset{B}{*}}A_2\right)\overset{e}{\underset{A_1}{*}}\left( A_1\overset{1}{\underset{B}{*}}A_3\right)\rightarrow\widetilde{A}.$$
\end{proof}

\noindent Using the previous identifications one can prove the following result about completely positive radial multipliers. For $\underline{i}=(i_1,\dots,i_n)\in I$ and $l\in\{1,2\}$ we define the number
$$\underline{i}_l=\vert\{s\in\{1,\dots,n\}\,:\, i_s=l\}\vert.$$
\begin{proposition}\label{PropMultipliers}
For all $k,l\in \{1,2\}$ and all $0< r\leq 1$ there exists a unique ucp map $\varphi_r\,:\,A_{v,k}\rightarrow  A_{v,k}$ such that $\varphi_r(\pi_k(b))=\pi_k(b)$ for all $b\in B$ and,
$$\varphi_r(\pi_k(a_1\dots a_n))=r^{\underline{i}_l}\pi_k(a_1\dots a_n)\text{ for all }a_1\dots a_n\in A_f\text{ reduced with }a_k\in A_{i_k}^\circ\text{ and }\underline{i}=(i_1,\dots,i_n).$$
\end{proposition}

\begin{proof}
We first prove the proposition for $k=1$. We separate the proof in two cases.
\vspace{0.2cm}

\noindent\textbf{Case 1: $l=2$.} Since $\pi_1$ is faithful on $A_1$, we may and will view $A_1\subset A_{v,1}$. After this identification, the canonical GNS-faithful ucp map $\EE_1\,:\,A_{v,1}\rightarrow A_1$ becomes a conditonal expectation. Consider the conditional expectation $\tau\ot \id\,:\, C([0,1])\ot B\rightarrow B$, where $\tau$ is the integral with respect to the normalized Lebesgue measure on $[0,1]$. We will also view $A_1\subset A_1\underset{B}{\overset{1}{*}}(C([0,1])\ot B)$ so that the canonical GNS-faithful ucp map $\widetilde{\EE}_1\,:\,A_1\underset{B}{\overset{1}{*}}(C([0,1])\ot B)\rightarrow A_1$ is a conditional expectation. Define $\widetilde{A}=A_{v,1}\overset{e}{\underset{A_1}{*}}\left(A_1\overset{1}{\underset{B}{*}}(C([0,1])\ot B)\right)$ with respect to the conditional expectations $\EE_1$ and $\widetilde{\EE}_1$. Since $\EE_1$ and $\widetilde{\EE}_1$ are GNS-faithful, the edge-reduced and the $k$-vertex-reduced amalgamated free products coincides for $k=1,2$. Hence, we may and will view $A_{v,1}\subset A_1\overset{1}{\underset{B}{*}}(C([0,1])\ot B)\subset \widetilde{A}$ and we have a canonical GNS-faithful conditional expectation $\widetilde{E}\,:\, \widetilde{A}\rightarrow A_{v,1}$. Also, by the first assertion of proposition \ref{CorVertex-EdgeReduced} we have a canonical identification $\widetilde{A}= A_1\overset{1}{\underset{B}{*}}\widetilde{A}_2$, where $\widetilde{A}_2=A_2\overset{e}{\underset{B}{*}}(C([0,1])\ot B)$. Let $\widetilde{\rho}_2\,:\,A_2\underset{B}{*} C([0,1])\ot B\rightarrow \widetilde{A}_2$ be the canonical surjection from the full to the edge-reduced amalgamated free product and $\widetilde{\pi}\,:\,A_1\underset{B}{*}\widetilde{A}_2\rightarrow A_1\overset{1}{\underset{B}{*}}\widetilde{A}_2=\widetilde{A}$  be the canonical surjection from the full to the vertex-reduced amalgamated free product. Fix $t\in\R$ and define the unitary $v_t\in C([0,1])$ by $v_t(x)=e^{2i\pi tx}$. Let $\rho_t=\vert\tau(v_t)\vert^2$ and $u_t=\widetilde{\pi}\circ\widetilde{\rho}_2(v_t\ot 1_B)\in \widetilde{A}$. Define the unital $*$-homomorphisms $\nu_1=\widetilde{\pi}\vert_{A_1}\,:\, A_1\rightarrow\widetilde{A}$ and $\nu_2\,:\,\widetilde{A}_2\rightarrow\widetilde{A}$ by $\nu_2(x)=u_t\widetilde{\pi}(x)u_t^*$. Note that $\nu_1$ is faithful. To simplify the notations we put $\widetilde{A}_1:=A_1$.
\vspace{0.2cm}

\noindent\textbf{Claim.}\textit{ For all $x=x_1\dots x_n\in A_1\underset{B}{*}\widetilde{A}_2$ reduced with $a_k\in \widetilde{A}_{i_k}^\circ$ and $\underline{i}=(i_1,\dots,i_n)$ one has:
$$\widetilde{\EE}(\nu_{i_1}(x_1)\dots\nu_{i_n}(x_n))=\left\{\begin{array}{lcl}
\rho_t^{\underline{i}_l}\widetilde{\pi}(x_1\dots x_n)&\text{if}&\widetilde{\pi}(x)\in A_{v,1},\\
0&\textit{if}&\widetilde{\EE}(\widetilde{\pi}(x))=0.\end{array}\right.$$}

\noindent\textit{Proof of the claim.} Note that $\widetilde{\pi}(x)\in A_{v,1}$ if and only if the letters $x_k$ of $x$ are alternating from $A_1^\circ$ and $\widetilde{\rho}_2(A_2^\circ)$ and $\widetilde{\EE}(\widetilde{\pi}(x))=0$ if and only if one of the letters of $x$ comes from $\rho_2((C([0,1])\ot B)^\circ)$. We prove the formula by induction on $n$. If $n=1$ we have either $x\in A_1^\circ$ in that case $\widetilde{\EE}(\nu_1(x))=\widetilde{\EE}(\widetilde{\pi}((x))=\widetilde{\pi}(x)$ or $x\in \widetilde{\rho}_2(\widetilde{A}_2^\circ)$ and

\begin{eqnarray*}
\widetilde{\EE}(\nu_2(x))&=&\widetilde{\EE}(u_t\widetilde{\pi}(x)u_t^*)\\
&=&\widetilde{\EE}((u_t-\tau(v_t))\widetilde{\pi}(x)(u_t^*-\overline{\tau(v_t)}))+\tau(v_t)\widetilde{\EE}(\widetilde{\pi}(x)(u_t^*-\overline{\tau(v_t)}))\\
&&+\overline{\tau(v_t)}\widetilde{\EE}((u_t-\tau(v_t))\widetilde{\pi}(x))+\vert\tau(v_t)\vert^2\widetilde{\EE}(\widetilde{\pi}(x))\\
&=&\vert\tau(v_t)\vert^2\widetilde{\EE}(\widetilde{\pi}(x))=\rho_t\widetilde{\EE}(\widetilde{\pi}(x)).
\end{eqnarray*}
\noindent Hence, $\widetilde{\EE}(\nu_2(x))=\left\{\begin{array}{lcl}\rho_t\widetilde{\pi}(x)&\text{if}&\widetilde{\pi}(x)\in A_{v,1},\\
0&\text{if}&\widetilde{\EE}(\widetilde{\pi}(x))=0.\end{array}\right.$

\noindent This proves the formula for $n=1$. Suppose that the formulae holds for a given $n\geq 1$. Let $x=x_1\dots x_{n+1}$ be reduced with $x_k\in \widetilde{A}_{i_k}^\circ$ and define $x'=x_1\dots x_n$ and $z=\nu_{i_1}(x_1)\dots\nu_{i_n}(x_n)$. Let $\underline{i}=(i_1,\dots ,i_{n+1})$ and $\underline{i}'=(i_1,\dots,i_{n})$.

\vspace{0.2cm}

\noindent Suppose that $x_{n+1}\in A_1^\circ$. Then $\underline{i}_2=\underline{i}'_2$ and,
$$\widetilde{\EE}(\nu_{i_1}(x_1)\dots\nu_{i_n}(x_n)\nu_{i_{n+1}}(x_{n+1}))=\widetilde{\EE}(\nu_{i_1}(x_1)\dots\nu_{i_n}(x_n)\widetilde{\pi}(x_{n+1}))=\widetilde{\EE}(z)\widetilde{\pi}(x_{n+1}).$$
Hence, if $\widetilde{\pi}(x)\in A_{v,1}$ then also $\widetilde{\pi}(x')\in A_{v,1}$ and we have, by the induction hypothesis, 
$$\widetilde{\EE}(\nu_{i_1}(x_1)\dots\nu_{i_n}(x_n)\nu_{i_{n+1}}(x_{n+1}))=\rho_t^{\underline{i}'_2}\widetilde{\pi}(x')\widetilde{\pi}(x_{n+1})=\rho_t^{\underline{i}_2}\widetilde{\pi}(x).$$
If $\widetilde{\EE}(\widetilde{\pi}(x))=0$ then also $\widetilde{\EE}(\widetilde{\pi}(x'))=0$ and we have, by the induction hypothesis, $\widetilde{\EE}(z)=0$ so $\widetilde{\EE}(\nu_{i_1}(x_1)\dots\nu_{i_n}(x_n)\nu_{i_{n+1}}(x_{n+1}))=0$.

\vspace{0.2cm}

\noindent Suppose now that $x_{n+1}\in \widetilde{A}_2^\circ$ then $x_n\in A_1^\circ$ and we have,
\begin{eqnarray*}
\widetilde{\EE}(z\nu_{i_{n+1}}(x_{n+1}))&=&\widetilde{\EE}(zu_t\widetilde{\pi}(x_{n+1})u_t^*)\\
&=&\widetilde{\EE}(z(u_t-\tau(v_t))\widetilde{\pi}(x_{n+1})(u_t^*-\overline{\tau(v_t)}))+\tau(v_t)\widetilde{\EE}(z\widetilde{\pi}(x_{n+1})(u_t^*-\overline{\tau(v_t)}))\\
&&+\overline{\tau(v_t)}\widetilde{\EE}(z(u_t-\tau(v_t))\widetilde{\pi}(x_{n+1}))+\vert\tau(v_t)\vert^2\widetilde{\EE}(z\widetilde{\pi}(x_{n+1}))\\
&=&\vert\tau(v_t)\vert^2\widetilde{\EE}(z\widetilde{\pi}(x_{n+1}))=\rho_t\widetilde{\EE}(z\widetilde{\pi}(x_{n+1})).
\end{eqnarray*}
\noindent Hence, if $\widetilde{\pi}(x)\in A_{v,1}$ then also $\widetilde{\pi}(x')\in A_{v,1}$ and $x_{n+1}\in A_2^\circ$ so $\widetilde{\pi}(x_{n+1})\in A_{v,1}$ and $\underline{i}_2=\underline{i}_2'+1$. By the preceding computation and the induction hypothesis we find:
$$\widetilde{\EE}(z\nu_{i_{n+1}}(x_{n+1}))=\rho_t\widetilde{\EE}(z\widetilde{\pi}(x_{n+1}))=\rho_t\widetilde{\EE}(z)\widetilde{\pi}(x_{n+1})=\rho_t\rho_t^{\underline{i}_2'}\widetilde{\pi}(x')\widetilde{\pi}(x_{n+1})=\rho_t^{\underline{i}_2}\widetilde{\pi}(x).$$
\noindent Finally, if $\widetilde{\EE}(\widetilde{\pi}(x))=0$, we need to prove that $\widetilde{\EE}(z\widetilde{\pi}(x_{n+1}))=0$. Note that, since $x_n\in A_1^\circ$, we have $z=\nu_{i_1}(x_1)\dots\nu_{i_{n-1}}(x_{n-1})x_n$. Hence, if $\widetilde{\EE}(\widetilde{\pi}(x'))=0$ so by the induction hypothesis we have $\widetilde{\EE}(z)=0$, $z$ may be written as a sum of reduced operators, containing at least one letter from $\widetilde{\rho}_2(( C([0,1])\ot B)^\circ)$ and ending with a letter from $A_1^\circ$. It follows that $z\widetilde{\pi}(x_{n+1})$ may be written as a sum of reduced operators, containing at least one letter from $\widetilde{\rho}_2(( C([0,1])\ot B)^\circ)$. Hence, $\widetilde{\EE}(z\widetilde{\pi}(x_{n+1}))=0$. Eventually, if $\widetilde{\EE}(\widetilde{\pi}(x))=0$ and $\widetilde{\EE}(\widetilde{\pi}(x'))\in A_{v,1}$ then, $x_1,\dots x_n\in A_1^\circ\cup A_2^\circ$ but $\widetilde{\EE}(\widetilde{\pi}(x_{n+1}))=0$. It follows that $z=\nu_{i_1}(x_1)\dots\nu_{i_{n-1}}(x_{n-1})x_n$ may be written as a sum of reduced operators ending with a letter from $A_1^\circ$. Hence, $z\widetilde{\pi}(x_{n+1})$ be be written as a sum of reduced operators containing at least one letter from $\widetilde{\rho}_2(( C([0,1])\ot B)^\circ)$. Hence, $\widetilde{\EE}(z\widetilde{\pi}(x_{n+1}))=0$.

\vspace{0.2cm}

\noindent\textit{End of the proof of the proposition.} By the claim, $\EE_1\circ\widetilde{\EE}(\nu_{i_1}(x_1)\dots\nu_{i_n}(x_n))=0$ for all reduced operators $x=x_1\dots x_n\in A_1\underset{B}{*}\widetilde{A}_2$ which are not in $A_1$ and, we obviously have,  $\EE_1\circ\widetilde{\EE}\circ\nu_1=\id_{A_1}$. Viewing $\widetilde{A}= A_1\overset{1}{\underset{B}{*}}\widetilde{A}_2$ and using the universal property of the vertex-reduced amalgamated free product, there exists, for all $t\in\R$, a unique unital $*$-isomorphism $\alpha_t\,:\,\widetilde{A}\rightarrow\widetilde{A}$ such that $\alpha_t(\widetilde{\pi}(a))=\widetilde{\pi}(a)$ if $a\in A_1$ and $\alpha_t((\widetilde{\pi}(x))=u_t\widetilde{\pi}(x)u_t^*$ if $x\in A_2\overset{e}{\underset{B}{*}}(C([0,1])\ot B)$. In particular, it follows from the claim that $\widetilde{\EE}\circ\alpha_t\vert_{A_{v,1}}\,:\,A_{v,1}\rightarrow A_{v,1}$, which is a ucp map, satisfies the properties of the map $\varphi_r$ described in the statement of the proposition, with $r=\rho_t=\left\vert\frac{\sin(\pi t)}{\pi t}\right\vert^2$. This concludes the proof. 
\vspace{0.2cm}

\noindent\textbf{Case 2: $l=1$.} The proof is similar. This time, the automorphism $\alpha_t\,:\,\widetilde{A}\rightarrow\widetilde{A}$ is defined, by the universal property, starting with the maps $\nu_1=\widetilde{\pi}\vert_{A_1}\,:\, A_1\rightarrow\widetilde{A}$ and $\nu_2\,:\,\widetilde{A}_2\rightarrow\widetilde{A}$ defined by $\nu_1(a)=u_t\widetilde{\pi}(a)u_t^*$ and $\nu_2(x)=\widetilde{\pi}(x)$. The reminder of the proof is the same.

\vspace{0.2cm}

\noindent The proof for $k=2$ is the same, using the second assertion of proposition \ref{CorVertex-EdgeReduced}.
\end{proof}

\section{$K$-equivalence between the full and reduced amalgamated free products}
\noindent Let $A_1$, $A_2$ be two unital C*-algebra with a common C*-subalgebra $B\subset A_k$, $k=1,2$ and denote by $A_f$ the full amalgamated free product.

\vspace{0.2cm}

\noindent Let $A:=A_1\overset{v}{\underset{B}{*}} A_2$ be the vertex-reduced amalgamated free product. For $k=1,2$, let $E_{A_k}$ (resp. $E_B$) be the canonical conditional expectation from $A$ to $A_k$ (reps. from $A$ to $B$). We will denote by the same symbol $\mathcal{A}$ the set of reduced operators viewed in $A$ or in $A_f$. Recall that the linear span of $\mathcal{A}$ and $B$ is a weakly dense unital $*$-subalgebra of $A$ (resp. $A_f$).

\vspace{0.2cm}

\noindent We denote by $\lambda\,:\,A_f\rightarrow A$ the canonical surjective unital $*$-homomorphism which is the identity on $\mathcal{A}$. In this section we prove the following result.

\begin{theorem}\label{TheoremKequivalence}
$[\lambda]\in {\rm KK}(A_f,A)$ is invertible.
\end{theorem}

\noindent The following lemma is well known (see \cite[Lemma 3.1]{Ve04}). We include a proof for the convenience of the reader.

\begin{lemma}\label{LemmaCE}
Let $n\geq 1$, $a_k\in A_{l_k}^{\circ}$ for $1\leq k\leq n$, and $a=a_1\dots a_n\in A$ a reduced word. One has
$$E_{A_k}(a^*a)=E_B(a^*a)\quad\text{whenever}\quad l_n\neq k.$$
\end{lemma}

\begin{proof}
We prove it for $k=1$ by induction on $n$. The proof for $k=2$ is the same.

\vspace{0.2cm}

\noindent It's obvious for $n=1$. Suppose that $n\geq 2$, define $b=E_B(a_1^*a_1)^{\frac{1}{2}}$, $x=(ba_2)\dots a_n$. One has:
$$E_{A_1}(a^*a)=E_{A_1}(a_n^*\dots a_1^*a_1\dots a_n)=E_{A_1}(a_n^*\dots a_2^*E_B(a_1^*a_1)a_2\dots a_n)=E_{A_1}(x^*x)=E_B(x^*x),$$
where we applied the induction hypothesis to get the last equality. Since the same computation gives $E_B(a^*a)=E_B(x^*x)$, this concludes the proof.
\end{proof}

\noindent We denote by $(H_k,\pi_k,\xi_k)$ (resp. $(K,\rho,\eta)$) the GNS construction of $E_{A_k}$ (resp. $E_B$). We may and will assume that $A\subset\mathcal{L}_{A_k}(H_k)$ and $\pi_k=\id$.

\vspace{0.2cm}

\noindent Observe that the Hilbert $A_k$-module $\xi_k.A_k\subset H_k$ is orthogonally complemented i.e. $H_k=\xi_k.A_k\oplus H_k^{\circ}$, as Hilbert $A_k$-modules, where $H_k^{\circ}$ is the closure of $\{a\xi_k\,:\,a\in A,\,\,E_{A_k}(a)=0\}$.

\vspace{0.2cm}

\noindent We now define a partial isometry $F_k\in\mathcal{L}_{A_k}(H_k,K\underset{B}{\ot} A_k)$ in the following way. First we put $F_k(\xi_k.a)=0$ for all $a\in A_k$. Then, it follows from lemma \ref{LemmaCE} that we can define an isometry $F_k\,:\, H_k^{\circ}\rightarrow K\underset{B}{\ot} A_k$ by the following formula:
$$F_k(a_1\dots a_n\xi_k)=\left\{\begin{array}{lcl}
\rho(a_1\dots a_n)\eta\underset{B}{\ot} 1&\text{if}&l_n\neq k\\
\rho(a_1\dots a_{n-1})\eta\underset{B}{\ot} a_n&\text{if}& l_n=k\end{array}\right.\quad\text{for all}\quad a_1\dots a_n\in A\,\,\text{a reduced operator}.$$

\noindent Hence, $F_k\in\mathcal{L}_{A_k}(H_k,K\underset{B}{\ot} A_k)$ is a well defined partial isometry such that $1-F_k^*F_k$ is the orthogonal projection onto $\xi_k.A_k$ and, $1-F_kF_k^*$ is the orthogonal projection onto
$$(\eta\ot 1).A_k\oplus\overline{\text{Span}}\{\rho(a_1\dots a_n)\eta\ot 1\,:\,a=a_1\dots a_n\in A\,\,\text{reduced with}\,\,l_n= k\}.A_k.$$

\noindent We will denote in the sequel $q_0$ the orthogonal projection of $K$ onto $\eta.B$, and for $l=1,2$ $q_l$ the projection in $K$ such that $F_lF_l^*=q_l\otimes_{A_l} 1$.
It is clear that $1=q_1+q_2+q_0$ and that all the projections commutes.  Define also $\overline{F}_l=F_l+\theta_{\eta\otimes_B 1, \xi_l}$. It is again clear that $\overline F_l$ is
an isometry and $\overline F_l \overline F_l^*=q_l+q_0=1-q_k$ for $k\neq l$.

\begin{lemma}\label{LemmaCompactCommutation}
For $k=1,2$ the following holds.
\begin{enumerate}
\item $\rho(a)F_k=F_ka\in\mathcal{L}_{A_k}(H_k,K\underset{B}{\ot} A_k)$ for all $a\in A_k$.
\item ${\rm Im}(\rho(a)F_k-F_ka)\subset(\rho(a)\eta\underset{B}{\ot}1).A_k\oplus(\eta\underset{B}{\ot}1).A_k$ for all $a\in A_l^{\circ}$ with $l\neq k$.
\item $\rho(x)F_k-F_kx\in\mathcal{K}_{A_k}(H_k,K\underset{B}{\ot} A_k)$ for all $x\in A$.
\item $\rho(a)\overline F_k=\overline F_k a$ $\forall a\in A_l$ with $l\neq k$ and $\rho(x)\overline F_k-\overline F_kx\in\mathcal{K}_{A_k}(H_k,K\underset{B}{\ot} A_k)$ $\forall x\in A$.

\end{enumerate}
\end{lemma}

\begin{proof}
We prove the lemma for $k=1$. The proof for $k=2$ is the same.

\vspace{0.2cm}

\noindent $(1)$. When $a\in B$ the commutation is obvious hence we may and will assume that $a\in A_1^{\circ}$. One has $F_1a\xi_1=0=\rho(a)F_1\xi_1$. Let now $n\geq 1$ and  $x=a_1\dots a_n\in A$, $a_k\in A_{l_k}^{\circ}$, be a reduced operator with $E_{A_1}(x)=0$. It suffices to show that $F_1ax\xi_1=\rho(a) F_1x\xi_1$. If $n=1$ we must have $x\in A_2^{\circ}$ and $F_1ax\xi_1=\rho(ax)\eta\ot 1 =\rho(a)F_1x\xi_1$. Suppose that $n\geq 2$. If $l_1=2$ then $ax$ is reduced and ends with a letter from $A_{l_n}^{\circ}$. It follows that $F_1ax\xi_2=\rho(a)F_1x\xi_2$. It $l_1=1$ then we can write $ax=(aa_1)^{\circ}a_2\dots a_n+ E_B(aa_1)a_2\dots a_n$. Since $a_2\dots a_n$ is reduced and ends with $l_n$ we find again that $F_1ax\xi_1=\rho(a)F_1x\xi_1$.

\vspace{0.2cm}

\noindent $(2)$. Let $a\in A_2^{\circ}$ and put $X_a=(\rho(a)\eta\underset{B}{\ot}1).A_k\oplus(\eta\underset{B}{\ot}1).A_k$. We have $F_1a\xi_1=\rho(a)\eta\ot 1$ and $\rho(a)F_1\xi_1=0$ hence, $(\rho(a)F_1-F_1a)\xi_1=-\rho(a)\eta\ot 1\in X_a$. Let now $n\geq 1$ and  $x=a_1\dots a_n\in A$, $a_k\in A_{l_k}^{\circ}$, be a reduced operator with $E_{A_1}(x)=0$. If $n=1$ we must have $x\in A_2^{\circ}$. It follows that $F_1ax\xi_1=F_1(ax)^{\circ}\xi_1+F_1E_B(ax)\xi_1=\rho((ax)^{\circ})\eta\ot 1$ and $\rho(a)F_1x\xi_1=\rho(ax)\eta\ot 1$. Hence, $(\rho(a)F_1-F_1a)x\xi_1=E_B(ax)\eta\ot 1=(\eta\ot 1).E_B(ax)\in X_a$. If $n\geq 2$, arguing as in the proof of $(1)$, we see that $F_1ax\xi_1=\rho(a)F_1x\xi_1$. Hence, ${\rm Im}(\rho(a)F_k-F_ka)\subset X_a$.
\vspace{0.2cm}

\noindent $(3)$. It is obvious since $A$ is generated, as a C*-algebra, by $A_1$ and $A_2^\circ$ and, by assertions $(1)$ $\rho(a)F_k-F_ka=0$ if $a\in A_k$ and the computation of $(2)$ shows that for $a\in A_2^\circ$, $\rho(a)F_k-F_ka= \theta(a)$ where $\theta(a)$ is the "rank one" operator that sends $\zeta \in H_k$ to $-\rho(a)\eta\otimes 1<\xi_1,\zeta>_{H_k}$. Hence  $\rho(a)F_k-F_ka\in\mathcal{K}_{A_k}(H_k,K\underset{B}{\ot} A_k)$ for all $a\in A_1\cup A_2^\circ$ and therefore for all $a\in A$.
\vspace{0.2cm}

\noindent $(4)$. The second part is obvious in view of $(3)$ as $\overline F_1$ is a compact perturbation of $F_1$, so let's concentrate on the exact commutation. Let $a\in A_2^{\circ}$. Clearly 
$\overline F_1a\xi_1=F_1a\xi_1=\rho(a)\eta\ot 1$ and $\rho(a)\overline F_1\xi_1=\rho(a)\eta\otimes 1$. 
 Let now $n\geq 1$ and  $x=a_1\dots a_n\in A$, $a_k\in A_{l_k}^{\circ}$, be a reduced operator with $E_{A_1}(x)=0$. 
 If $n=1$ we must have $x\in A_2^{\circ}$. It follows that $\overline F_1ax\xi_1=F_1(ax)^{\circ}\xi_1+\theta_{\eta\otimes 1, \xi_1}E_B(ax)\xi_1=\rho((ax)^{\circ})\eta\ot 1+E_B(ax)\eta\ot 1 $ and $\rho(a)\overline F_1x\xi_1=F_1x\xi_1=\rho(ax)\eta\ot 1$. If $n\geq 2$, arguing as in the proof of $(1)$, we see that $\overline F_1ax\xi_1=F_1ax\xi_1=\rho(a)F_1x\xi_1=\rho(a)\overline F_1x\xi_1$.
 \end{proof}
\noindent We define the following Hilbert $A_f$-modules:
$$H_m=H_1\underset{A_1}{\ot}A_f\oplus H_2\underset{A_2}{\ot}A_f\quad\text{and}\quad
K_m=K\underset{B}{\ot}A_f=\left(K\underset{B}{\ot}A_k\right)\underset{A_k}{\ot} A_f,$$
with the canonical representations $\pi\,:\,A\rightarrow\mathcal{L}_{A_f}(H_m)$, $\pi(x)=x\underset{A_1}{\ot}1_{A_f}\oplus x\underset{A_2}{\ot}1_{A_f}$ and $\bar\rho\,:\,A\rightarrow\mathcal{L}_{A_f}(K)$, $\bar\rho(x)=\rho(x)\underset{B}{\ot}1_{A_f}$.
We consider, for $k=1,2$, the partial isometry
$$F_k\underset{A_k}{\ot}1_{A_f}\in\mathcal{L}_{A_f}(H_k\underset{A_k}{\ot}A_f,(K\underset{B}{\ot}A_k)\underset{A_k}{\ot} A_f).$$
Observe that $F_1\underset{A_1}{\ot}1_{A_f}$ and $F_2\underset{A_2}{\ot}1_{A_f}$ have orthogonal images. Indeed, the image of $F_k\underset{A_k}{\ot}1_{A_f}$ is the closed linear span of $\{\rho(a_1\dots a_n)\eta\underset{B}{\ot} y\,:\,y\in A_f\,\,\text{and}\,\,a_1\dots a_n\in A\,\,\text{reduced with}\,\,a_n\notin A_{k}^{\circ}\}$. Hence the operator $F\in\mathcal{L}_{A_f}(H_m,K_m)$ defined by $F=F_1\underset{A_1}{\ot}1_{A_f}\oplus F_2\underset{A_2}{\ot}1_{A_f}$ is a partial isometry such that $1-FF^*$ is the orthogonal projection onto $(\eta\underset{B}{\ot} 1_{A_f}).A_f$ and $1-F^*F$ is the orthogonal projection onto $(\xi_1\underset{A_1}{\ot}1_{A_f}).A_f\oplus(\xi_2\underset{A_2}{\ot}1_{A_f}).A_f$. In particular $1-F^*F, 1-FF^*\in\mathcal{K}_{A_f}(H_m,K_m)$. Moreover, it follows from lemma \ref{LemmaCompactCommutation} that $F\pi(x)-\bar\rho(x)F\in\mathcal{K}_{A_f}(H_m,K_m)$ for all $x\in A$. Hence, we get an element $\alpha=[(H_m\oplus K_m,\pi\oplus\bar\rho,F)]\in {\rm KK}(A,A_f)$.

\vspace{0.2cm}

\noindent To prove theorem \ref{TheoremKequivalence} it suffices to prove that $\alpha\underset{A_f}{\ot}[\lambda]=[\id_{A}]$ in ${\rm KK}(A,A)$ and $[\lambda]\underset{A}{\ot}\alpha=[\id_{A_f}]$ in ${\rm KK}(A_f,A_f)$. We prove the easy part in the next proposition.

\begin{proposition}\label{sub-K-equivalence}
One has $[\lambda]\underset{A}{\ot}\alpha=[\id_{A_f}]$ in ${\rm KK}(A_f,A_f)$.
\end{proposition}

\begin{proof}
Observe that $[\lambda]\underset{A}{\ot}\alpha=[(H_m\oplus K_m,\pi_m\oplus\rho_m,F)]$ where $\pi_m=\pi\circ\lambda\,:\,A_f\rightarrow\mathcal{L}_{A_f}(H_m)$ and $\rho_m=\bar\rho\circ\lambda\,:\,A_f\rightarrow\mathcal{L}_{A_f}(K_m)$. Hence, $[\lambda]\underset{A}{\ot}\alpha-[\id_{A_f}]$ is represented by the Kasparov triple $(H_m\oplus\widetilde{K}_m,\pi_m\oplus\widetilde{\rho}_m,\widetilde{F})$, where $\widetilde{K}_m=K_m\oplus A_f$ and $\widetilde{\rho}_m(x)=\rho_m(x)\oplus x$, where we view $A_f=\mathcal{L}_{A_f}(A_f)$ by left multiplication. Finally, $\widetilde{F}\in\mathcal{L}_{A_f}(H_m,\widetilde{K}_m)$ is the unitary defined by
$$\widetilde{F}(\xi_1\underset{A_1}{\ot}1_{A_f})=\eta\underset{B}{\ot}1_{A_f},\quad\widetilde{F}(\xi_2\underset{A_2}{\ot}1_{A_f})=1_{A_f}\quad\text{and,}$$
$$\widetilde{F}(\xi)=F(\xi)\,\,\text{for all}\,\,\xi\in H_m\ominus\left((\xi_1\underset{A_1}{\ot}1_{A_f}).A_f\oplus(\xi_2\underset{A_2}{\ot}1_{A_f}).A_f\right).$$
We collect some computations in the following claim.

\vspace{0.2cm}

\noindent\textbf{Claim.}\textit{ Let $v\in\mathcal{L}_{A_f}(H_m)$ be the self-adjoint unitary defined by the identity on $H_m\ominus((\xi_1\underset{A_1}{\ot} 1_{A_f}).A_f\oplus (\xi_2\underset{A_2}{\ot} 1_{A_f}).A_f)$ and $v(\xi_1\underset{A_1}{\ot} 1_{A_f})=\xi_2\underset{A_2}{\ot} 1_{A_f}$, $v(\xi_2\underset{A_2}{\ot} 1_{A_f})=\xi_1\underset{A_1}{\ot} 1_{A_f}$. One has:
\begin{enumerate}
\item $\widetilde{F}^*\widetilde{\rho}_m(b)\widetilde{F}=\pi_m(b)$ and $v^*\pi_m(b)v=\pi_m(b)$ for all $b\in B$.
\item $\widetilde{F}^*\widetilde{\rho}_m(a)\widetilde{F}=v^*\pi_m(a)v$ for all $a\in A_1$.
\item $\widetilde{F}^*\widetilde{\rho}_m(a)\widetilde{F}=\pi_m(a)$ for all $a\in A_2$.
\end{enumerate}}
\noindent\textit{Proof of the claim.}The proof of $(1)$ is obvious and we leave it to the reader. 

\vspace{0.2cm}
\noindent$(2)$. By $(1)$, it suffices to prove $(2)$ for $a\in A_1^{\circ}$. Let $a\in A_1^{\circ}$. One the one hand:

$$\widetilde{F}^*\widetilde{\rho}_m(a)\widetilde{F}\xi_1\underset{A_1}{\ot}1_{A_f}=\widetilde{F}^*(\rho(a)\eta\underset{B}{\ot}1_{A_f})=a\xi_2\underset{A_2}{\ot}1_{A_f}\quad\text{and}\quad
\widetilde{F}^*\widetilde{\rho}_m(a)\widetilde{F}\xi_2\underset{A_2}{\ot}1_{A_f}=\widetilde{F}^*(a)=\xi_2\underset{A_2}{\ot}a.$$

\noindent One the other hand:
$$v^*\pi_m(a)v\xi_1\underset{A_1}{\ot}1_{A_f}=v^*(a\xi_2\underset{A_2}{\ot} 1_{A_f})=a\xi_2\underset{A_2}{\ot} 1_{A_f}\quad\text{and}\quad
v^*\pi_m(a)v\xi_2\underset{A_2}{\ot}1_{A_f}=v^*(a\xi_1\underset{A_1}{\ot} 1_{A_f})=\xi_2\underset{A_2}{\ot} a.$$
\noindent Let now $x=a_1\dots a_n\in A$ be reduced operator with $a_k\in A_{l_k}^{\circ}$. We prove by induction on $n$ that $\widetilde{F}^*\widetilde{\rho}_m(a)\widetilde{F}x\xi_k\underset{A_k}{\ot} 1_{A_f}=v^*\pi_m(a)vx\xi_k\underset{A_k}{\ot} 1_{A_f}$ for all $k\in\{1,2\}$. Suppose that $n=1$ so $x\in A_1^{\circ}\cup A_2^{\circ}$ and let $k\in\{1,2\}$ such that $x\notin A_k^{\circ}$ (the case $x\in A_k^{\circ}$ has been done before). We have:

$$\widetilde{F}^*\widetilde{\rho}_m(a)\widetilde{F}x\xi_k\underset{A_k}{\ot} 1_{A_f}=\widetilde{F}^*(\rho(ax)\eta\underset{B}{\ot}1_{A_f})=\left\{\begin{array}{lcl}
(ax)^{\circ}\xi_2\underset{A_2}{\ot} 1_{A_f}+\xi_1\underset{A_1}{\ot}E_B(ax)&\text{if}&x\in A_{1}^{\circ},\\
ax\xi_1\underset{A_1}{\ot} 1_{A_f}&\text{if}&x\in A_{2}^{\circ}.\end{array}\right.$$
One the other hand we have:
$$v^*\pi_m(a)vx\xi_k\underset{A_k}{\ot} 1_{A_f}=v^*(ax\xi_k\underset{A_k}{\ot}1_{A_f})=\left\{\begin{array}{lcl}
(ax)^{\circ}\xi_2\underset{A_2}{\ot} 1_{A_f}+\xi_1\underset{A_1}{\ot}E_B(ax)&\text{if}&x\in A_{1}^{\circ}\,\,(k=2),\\
ax\xi_1\underset{A_1}{\ot} 1_{A_f}&\text{if}&x\in A_{2}^{\circ}\,\,(k=1).\end{array}\right.$$

\vspace{0.2cm}

\noindent Finally, suppose that $n\geq 2$ and the formula holds for $n-1$. Write $ax=y+z$, where, if $l_1=1$, $y=(aa_1)^{\circ}a_2\dots a_n$ and $z=E_B(aa_1)a_2\dots a_n$ and, if $l_1=2$, $y=ax$ and $z=0$. Observe that, in both cases,  $y$ is a reduced operator ending with a letter from $A_{l_n}^\circ$ and $z$ is either $0$ or a reduced operator ending with a letter from $A_{l_n}^\circ$. By the induction hypothesis, we may and will assume that $k\neq l_n$. We have:
\begin{eqnarray*}
\widetilde{F}^*\widetilde{\rho}_m(a)\widetilde{F}x\xi_k\underset{A_k}{\ot} 1_{A_f}&=&\widetilde{F}^*(\rho(ax)\eta\underset{B}{\ot} 1_{A_f})=\widetilde{F}^*(\rho(y)\eta\underset{B}{\ot} 1_{A_f})+\widetilde{F}^*(\rho(z)\eta\underset{B}{\ot} 1_{A_f})\\
&=&y\xi_k\underset{A_k}{\ot} 1_{A_f}+z\xi_k\underset{A_k}{\ot} 1_{A_f}=ax\xi_k\underset{A_k}{\ot} 1_{A_f}.
\end{eqnarray*}
Moreover,
\begin{eqnarray*}
v^*\pi_m(a)v x\xi_k\underset{A_k}{\ot} 1_{A_f}&=&v^*(ax\xi_k\underset{A_k}{\ot} 1_{A_f})=v^*(y\xi_k\underset{A_k}{\ot} 1_{A_f})+v^*(z\xi_k\underset{A_k}{\ot} 1_{A_f})\\
&=&y\xi_k\underset{A_k}{\ot} 1_{A_f}+z\xi_k\underset{A_k}{\ot} 1_{A_f}=ax\xi_k\underset{A_k}{\ot} 1_{A_f}.
\end{eqnarray*}

\noindent The proof of $(3)$ is similar.\hfill{$\Box$}

\vspace{0.2cm}

\noindent\textit{End of the proof of proposition \ref{sub-K-equivalence}.} Let $t\in\R$ and define $v_t=\cos(t)+iv\sin(t)\in\mathcal{L}_{A_f}(H_m)$. Since $v=v^*$ is unitary, $v_t$ is a unitary for all $t\in\R$. Moreover, assertion $(1)$ of the Claim implies that $v_t\pi_m(b)v_t^{*}=\pi_m(b)$ for all $b\in B$. It follows from the universal property of $A_f$ that there exists a unique unital $*$-homomorphism $\pi_t\,:\,A_f\rightarrow \mathcal{L}_{A_f}(H_m)$ such that:
$$\pi_t(a)=\left\{\begin{array}{lcl}
v_t^*\pi_m(a)v_t&\text{if}&a\in A_1,\\
\pi_m(a)&\text{if}&a\in A_2.\end{array}\right.$$
Then the triple $\alpha_t=(H_m\oplus\widetilde{K}_m,\pi_t\oplus\widetilde{\rho}_m,\widetilde{F})$ gives an homotopy between $\alpha_0$ which represents $[\lambda]\underset{A}{\ot}\alpha-[\id_{A_f}]$ and $\alpha_{\frac{\pi}{2}}$ which is degenerated by the claim.
\end{proof}

\noindent We finish the proof of theorem \ref{TheoremKequivalence} in the next proposition.

\begin{proposition}\label{PropositionKequivalence}
One has $\alpha\underset{A_f}{\ot}[\lambda]=[\id_{A}]$ in ${\rm KK}(A,A)$.
\end{proposition}

\begin{proof}
Observe that $\alpha\underset{A_f}{\ot}[\lambda]=[(H_r\oplus K_r,\pi_r\oplus\rho_r,F_r)]$ where
$$H_r=H_m\underset{\lambda}{\ot}A=H_1\underset{A_1}{\ot}A\oplus H_2\underset{A_2}{\ot}A\quad\text{and}\quad
K_r=K_m\underset{\lambda}{\ot}A=K\underset{B}{\ot}A=\left(K\underset{B}{\ot}A_k\right)\underset{A_k}{\ot} A,$$
with the canonical representations $\pi_r\,:\,A\rightarrow\mathcal{L}_{A}(H_r)$, $\pi_r(x)=\pi(x)\underset{\lambda}{\ot}1=x\underset{A_1}{\ot}1_{A}\oplus x\underset{A_2}{\ot}1_{A}$ and $\rho_r\,:\,A\rightarrow\mathcal{L}_{A}(K_r)$, $\rho_r(x)=\bar\rho(x)\underset{\lambda}{\ot}1=\rho(x)\underset{B}{\ot}1_{A}$ and with the operator $F_r=F\underset{\lambda}{\ot} 1\in\mathcal{L}_A(H_r,K_r)$. Hence, $\alpha\underset{A_f}{\ot}[\lambda]-[\id_{A}]$ is represented by the Kasparov triple $(H_r\oplus\widetilde{K}_r,\pi_r\oplus\widetilde{\rho}_r,\widetilde{F}_r)$, where $\widetilde{K}_r=K_r\oplus A$ and $\widetilde{\rho}_r(x)=\rho_r(x)\oplus x$, where we view $A=\mathcal{L}_{A}(A)$ by left multiplication. Finally, $\widetilde{F}_r\in\mathcal{L}_{A}(H_r,\widetilde{K}_r)$ is the unitary defined by
$$\widetilde{F}_r(\xi_1\underset{A_1}{\ot}1_{A})=\eta\underset{B}{\ot}1_{A},\quad\widetilde{F}_r(\xi_2\underset{A_2}{\ot}1_{A})=1_{A}\quad\text{and,}$$
$$\widetilde{F}(\xi)=F(\xi)\,\,\text{for all}\,\,\xi\in H_r\ominus\left((\xi_1\underset{A_1}{\ot}1_{A}).A\oplus(\xi_2\underset{A_2}{\ot}1_{A}).A\right).$$
The claim in the proof of proposition \ref{sub-K-equivalence} implies the following claim.

\vspace{0.2cm}

\noindent\textbf{Claim.}\textit{ Let $u\in\mathcal{L}_{A}(H_r)$ be the self-adjoint unitary defined by the identity on $H_r\ominus((\xi_1\underset{A_1}{\ot} 1_{A}).A\oplus (\xi_2\underset{A_2}{\ot} 1_{A}).A)$ and $u(\xi_1\underset{A_1}{\ot} 1_{A})=\xi_2\underset{A_2}{\ot} 1_{A}$, $u(\xi_2\underset{A_2}{\ot} 1_{A})=\xi_1\underset{A_1}{\ot} 1_{A}$. One has:
\begin{enumerate}
\item $\widetilde{F}^*\widetilde{\rho}_r(b)\widetilde{F}=\pi_r(b)$ and $u^*\pi_r(b)u=\pi_r(b)$ for all $b\in B$.
\item $\widetilde{F}^*\widetilde{\rho}_r(a)\widetilde{F}=u^*\pi_r(a)u$ for all $a\in A_1$.
\item $\widetilde{F}^*\widetilde{\rho}_r(a)\widetilde{F}=\pi_r(a)$ for all $a\in A_2$.
\end{enumerate}}
\noindent Let $t\in\R$ and define the unitary $u_t=\cos(t)+iu\sin(t)\in\mathcal{L}_{A}(H_r)$. Assertion $(1)$ of the Claim implies that $u_t^*\pi_r(b)u_t=\pi_r(b)$ for all $b\in B$. By the universal property of full amalgamated free products, for all $t\in\R$, there exists a unique unital $*$-homomorphism $\pi_t\,:\,A_f\rightarrow \mathcal{L}_{A}(H_r)$ such that:
$$\pi_t(a)=\left\{\begin{array}{lcl}
u_t^*\pi_r(a)u_t&\text{if}&a\in A_1,\\
\pi_r(a)&\text{if}&a\in A_2.\end{array}\right.$$
Arguing as in the end of the proof of proposition \ref{sub-K-equivalence}, we see that it suffices to show that, for all $t\in[0,\frac{\pi}{2}]$, $\pi_t$ factorizes through $A$ i.e. $\ker(\lambda)\subset\ker(\pi_t)$. To do that, we need the following claim.

\vspace{0.2cm}

\noindent\textbf{Claim.} \textit{For all $t\in\R$ and all $a=a_1\dots a_n\in \mathcal{A}$ a reduced operator with $a_k\in A_{l_k}^{\circ}$ one has
\begin{enumerate}
\item $\pi_t(a)u_t^*(\xi_2\underset{A_2}{\ot}1_A)=e^{-it}(a\xi_2\underset{A_2}{\ot}1_A)$ if $l_n=1$ and $\pi_t(a)(\xi_1\underset{A_1}{\ot}1_A)=a\xi_1\underset{A_1}{\ot}1_A$ if $l_n=2$.
\item $\langle u_t^{*}(\xi_1\underset{A_1}{\ot} 1_A),\pi_t(a) u_t^{*}(\xi_1\underset{A_1}{\ot} 1_A)\rangle=\sin^{2k}(t)a$ where $k=\left\{\begin{array}{ll}
\frac{n}{2}&\text{if }n\text{ is even},\\
\frac{n-1}{2}&\text{if }n\text{ is odd and }l_n=1,\\
\frac{n+1}{2}&\text{if }n\text{ is odd and }l_n=2.\end{array}\right.$
\item $\langle \xi_2\underset{A_2}{\ot} 1_A,\pi_t(a) \xi_2\underset{A_2}{\ot} 1_A\rangle=\sin^{2k}(t)a$ where $k=\left\{\begin{array}{ll}
\frac{n}{2}&\text{if }n\text{ is even},\\
\frac{n+1}{2}&\text{if }n\text{ is odd and }l_n=1,\\
\frac{n-1}{2}&\text{if }n\text{ is odd and }l_n=2.\end{array}\right.$
\end{enumerate}}

\noindent\textit{Proof of the claim.}
$(1)$ is obvious by induction on $n$ once observed that $u_t\xi=e^{it}\xi$ (and $u_t^*\xi=e^{-it}\xi$) for all $\xi\in H_r\ominus(\xi_1\underset{A_1}{\ot}1_A.A\oplus\xi_2\underset{A_2}{\ot}1_A.A)$.

\noindent $(2)$. Define, for $a_1\dots a_n\in\mathcal{A}$, $F(a_1,\dots,a_n)=\langle u_t^{*}(\xi_1\underset{A_1}{\ot} 1_A),\pi_t(a) u_t^{*}(\xi_1\underset{A_1}{\ot} 1_A)\rangle$. First suppose that $a\in A_1^{\circ}$ then $F(a)=\langle u_t^{*}(\xi_1\underset{A_1}{\ot} 1_A),u_t^*\pi_r(a)(\xi_1\underset{A_1}{\ot} 1_A)\rangle=\langle \xi_1\underset{A_1}{\ot} 1_A,\xi_1\underset{A_1}{\ot} a\rangle=a$. Now, let $a=a_1\dots a_n\in\mathcal{A}$ with $n\geq 2$ and $l_n=1$. We have:
$$F(a_1,\dots,a_n)
=\langle u_t^{*}(\xi_1\underset{A_1}{\ot} 1_A),\pi_t(a_1\dots a_{n-1})u_t^*(\xi_1\underset{A_1}{\ot} a_n)\rangle=F(a_1,\dots,a_{n-1})a_n.$$
Hence, it suffices to show the formula for $l_n=2$. Suppose $a\in A_2^{\circ}$, we have:
\begin{eqnarray*}
F(a)&=&\langle u_t^{*}(\xi_1\underset{A_1}{\ot} 1_A),\pi_r(a)u_t^*(\xi_1\underset{A_1}{\ot} 1_A)\rangle\\
&=&\langle \cos(t)\xi_1\underset{A_1}{\ot} 1_A-i\sin(t)\xi_2\underset{A_2}{\ot} 1_A,\cos(t)a\xi_1\underset{A_1}{\ot} 1_A-i\sin(t)\xi_2\underset{A_2}{\ot} a\rangle
=\sin^2(t)a.
\end{eqnarray*}
Now suppose $a_1a_2\in\mathcal{A}$, with $l_2=2$, $l_1=1$. We have:
\begin{eqnarray*}
F(a_1,a_2)&=&\langle \xi_1\underset{A_1}{\ot} 1_A,\pi_r(a_1)u_t\pi_r(a_2)u_t^*(\xi_1\underset{A_1}{\ot} 1_A)\rangle\\
&=&\langle \xi_1\underset{A_1}{\ot} 1_A,\pi_r(a_1)u_t(\cos(t)a_2\xi_1\underset{A_1}{\ot} 1_A-i\sin(t)\xi_2\underset{A_2}{\ot} a_2)\rangle\\
&=&\langle \xi_1\underset{A_1}{\ot} 1_A,\cos(t)e^{it}a_1a_2\xi_1\underset{A_1}{\ot} 1_A-i\cos(t)\sin(t) a_1\xi_2\underset{A_2}{\ot} a_2+\sin^2(t)\xi_1\underset{A_1}{\ot} a_1a_2\rangle\\
&=&\sin^2(t)a_1a_2.
\end{eqnarray*}
Finally, suppose that $n\geq 3$ and $a_1\dots a_n\in\mathcal{A}$ with $l_n=2$. Define $x=a_1\dots a_{n-2}$. We have
\begin{eqnarray*}
F(a_1,\dots,a_n)&=&\langle u_t^{*}(\xi_1\underset{A_1}{\ot} 1_A),\pi_t(x)u_t^*\pi_r(a_{n-1})u_t\pi_r(a_n) u_t^{*}(\xi_1\underset{A_1}{\ot} 1_A)\rangle\\
&=&\langle u_t^{*}(\xi_1\underset{A_1}{\ot} 1_A),\pi_t(x)u_t^*\pi_r(a_{n-1})u_t(\cos(t)a_n\xi_1\underset{A_1}{\ot} 1_A-i\sin(t)\xi_2\underset{A_2}{\ot} a_n)\rangle
\end{eqnarray*}
$$
=\langle u_t^{*}(\xi_1\underset{A_1}{\ot} 1_A),\pi_t(x)u_t^*(\cos(t)e^{it}a_{n-1}a_n\xi_1\underset{A_1}{\ot} 1_A-i\cos(t)\sin(t) a_{n-1}\xi_2\underset{A_2}{\ot} a_n+\sin^2(t)\xi_1\underset{A_1}{\ot} a_{n-1}a_n)\rangle
$$
$$
=\langle u_t^{*}(\xi_1\underset{A_1}{\ot} 1_A),\cos(t)a_1\dots a_n\xi_1\underset{A_1}{\ot} 1_A-ie^{-it}\cos(t)\sin(t)a_1\dots a_{n-1}\xi_2\underset{A_2}{\ot} a_n\rangle$$
$$
+\langle u_t^{*}(\xi_1\underset{A_1}{\ot} 1_A),\sin^2(t)\pi_t(x)u_t^*\xi_1\underset{A_1}{\ot} a_{n-1}a_n)\rangle.
$$
Hence we find:
$$F(a_1,\dots,a_n)=\sin^2(t)\langle u_t^{*}(\xi_1\underset{A_1}{\ot} 1_A),\pi_t(x)u_t^*\xi_1\underset{A_1}{\ot} a_{n-1}a_n)\rangle=\sin^2(t)F(a_1,\dots,a_{n-2})a_{n-1}a_n.$$
The result now follows by an obvious induction on $n$. The proof of $(3)$ is similar.\hfill{$\Box$}
\vspace{0.2cm}

\noindent\textit{End of the proof of proposition \ref{PropositionKequivalence}.} Fix $t\in[0,\frac{\pi}{2}]$ and let $A_t$ be the C*-subalgebra of $\mathcal{L}_A(H_r)$ generated by $\pi_t(A_1)\cup\pi_t(A_2)$. Hence, $\pi_t\,:\,A_f\rightarrow A_t$ is surjective. Consider the ucp map $\varphi_t\,:\,A_t\rightarrow A$ defined by $\varphi_t(x)=\frac{1}{2}\left(\langle u_t^*(\xi_1\underset{A_1}{\ot}1_A),xu_t^*(\xi_1\underset{A_1}{\ot}1_A)\rangle+\langle \xi_2\underset{A_2}{\ot}1_A,x\xi_2\underset{A_2}{\ot}1_A\rangle\right)$ and note that $\varphi_t$ is GNS faithful. Indeed, let $x\in A_t$ such that $\varphi_t(y^*x^*xy)=0$ for all $y\in A_t$. Then $L\subset\ker(x)$ where,
\begin{eqnarray*}
L&=&\overline{\text{Span}}\left(A_tu_t^*(\xi_1\underset{A_1}{\ot} 1_A).A\cup A_t(\xi_2\underset{A_2}{\ot} 1_A).A\right)=\overline{\text{Span}}\left(A_t(\xi_1\underset{A_1}{\ot} 1_A).A\cup A_t(\xi_2\underset{A_2}{\ot} 1_A).A \right)\\
&=&\overline{\text{Span}}\left(A_t(\xi_1\underset{A_1}{\ot} 1_A).A\cup A_tu_t^*(\xi_2\underset{A_2}{\ot} 1_A).A \right)=H_r,
\end{eqnarray*}
where we used assertion $(3)$ of the claim for the last equality. Hence $x=0$. Let $A_{v,k}$ for $k=1,2$ be the $k$-vertex-reduced free product and call $i_k$ the natural inclusion of $A$ in $A_{v,k}$ and $\pi_k=i_k\circ\lambda$ the natural map from $A_f$ to $A_{v,k}$. Clearly $||x||_A=\max (||i_1(x)||,||,i_2(x)||)$ for any $x$ in the vertex-reduced free product $A$. From the assertions $(1)$ and $(2)$ of the claim and proposition \ref{PropMultipliers} with $r=\sin^2(t)$ we deduced that for any  $k=1,2$ there exists two ucp maps $\psi_1^k$ and $\psi_2^k$ from $A_{v,k}$ to
itself such that $i_k(\varphi_t(\pi_t(a)))=\frac{1}{2} (\psi_1^k(\pi_k(a))+\psi_2^k(\pi_k(a)))$ for all $a\in A_f$. Therefore $||\varphi_t(\pi_t(a))||_A \leq  \max (||\pi_1(a)||, ||\pi_2(a)||) =|| \lambda(a) ||$ for all $a\in A_f$. Let us show that $\ker(\lambda)\subset\ker(\pi_t)$. Let $x\in\ker(\lambda)$. Then, for all $y\in A_f$ we have $\lambda(y^*x^*xy)=0$. Therefore  $\varphi_t\circ\pi_t(y^*x^*xy)=0$ for all $y\in A_f$. Since $\pi_t$ is surjective we deduced that $\varphi_t(y^*\pi_t(x)^*\pi_t(x)y)=0$ for all $y\in A_t$. Using that $\varphi_t$ is GNS faithful we deduce that $\pi_t(x)=0$.
\end{proof}
\noindent We obtain the following obvious corollary of theorem \ref{TheoremKequivalence} and corollary \ref{CorDegVertexRed}.
\begin{corollary}[\cite{Cu82}] If we have conditional expectations $E_k\,:\, A_k\rightarrow B$ which are also unital $*$-homomorphism, then 
the canonical surjection $A_1\underset{B}{*} A_2\rightarrow A_1\oplus_B A_2$ is $K$-invertible 
\end{corollary}

\section{A long exact sequence in $KK$-theory for full amalgamated free products}

\noindent Let $A_1$ and $A_2$ two unital C*-algebras with a common unital C*-subalgebra $B$.  We will denote by $i_l$ the inclusion of $B$ in $A_l$ for $l=1,2$. The algebra
$A_f$ is the full amalgamated free product. To simplify notation we will denote by $S$ the algebra $C_0(]-1,1[)$.
\vspace{0.2cm}

\noindent Let $D$ be the subalgebra of $S\otimes A_f$ consisting of functions $f$ such that $f(]-1,0[)\subset A_1$, $f(]0,1[)\subset A_2$ and $f(0)\in B$. This algebra is of course
isomorphic to the cone of $i_1\oplus i_2$ from $B$ to $A_1\oplus A_2$.
We call   $j$ the inclusion of $D$ in the suspension of $A_f$.

\begin{theorem}\label{bigthm}
Suppose that there exist unital conditional expectations from $A_l$ to $B$ for $l=1,2$, then the map $j$, seen as an element $[j]$ of $KK^0(D, S\otimes A_f)$, is invertible.
\end{theorem}

\noindent The proof of this result will be done in several steps. We will start with the construction of an element $x$ of $KK^1(A_f, D)$. As $KK^1(A_f,D)$ is isomorphic to $KK^0( 
S\otimes A_f, D)$ this will produce a candidate  $y$ for the inverse of $j$. The proof that $y\otimes_D [j]$ is the identity of the suspension of $A_f$ will use \ref{sub-K-equivalence}.
Finally the proof that $[j]\otimes_{S\otimes A_f}y$ is the identity of $D$ will be done indirectly by using a short exact sequence for $D$.

\subsection{An inverse in KK-theory}

In order to present the inverse, we need some additional notations and preliminaries. Let $\kappa_1$ be the inclusion of $C_0(]-1,0[;A_1)$ in $D$ and $\kappa_2$ the inclusion of $C_0(]0,1[;A_2)$ in $D$.  There is also $\kappa_0$ the obvious map from $S\otimes B$ in $D$.
As $K$ of the preceding section is a $B$-module, we can define
$$K_0=(K\otimes S)\otimes_{\kappa_0} D,\,\,K_1=(K\otimes_{i_1}A_1\otimes C_0(]-1,0[))\otimes_{\kappa_1} D\text{ and }K_2=(K\otimes_{i_2}A_2\otimes C_0(]0,1[))\otimes_{\kappa_2} D.$$
\noindent If one defines $I_l$ as the images of $\kappa_l$ in $D$ for $l=1,2$, it is clear that these are ideals in $D$.

\begin{lemma}\label{lemmaK}
$K_l$ is isomorphic to $\overline{K_0.I_l}$ for $l=1,2$ as $D$  Hilbert module.
\end{lemma}

\begin{proof}
Lets do if for $l=1$. Indeed as $I_1=\overline{C_0(]-1,0[).I_1}$ because an approximate unit for $C_0(]-1,0[)$ is also one for $I_1$, it is easy to see that $\overline{K_0.I_1}$ is  isomorphic to $\overline{(K\otimes S).C_0(]-1,0[) }\otimes_{\kappa_0} D .I_1$, i.e. $(K\otimes C_0(]-1,0[)) \otimes_{\kappa_0} D .I_1$. Considering that $C_0(]-1,0[;A_1)\otimes_{\kappa_1} D$ is $D.I_1$, one gets that $\overline{K_0.I_1}$ is nothing but $(K\otimes C_0(]-1,0[)) \otimes_{\tilde\kappa_0} C_0(]-1,0[;A_1)\otimes_{\kappa_1} D$ where $\tilde\kappa_0$ is the natural inclusion of $C_0(]-1,0[;B)$ in $C_0(]-1,0[;A_1)$, i.e. $i_1\otimes Id_{C_0(]-1,0[)}.$ Therefore
 $(K\otimes_{i_1}A_1)\otimes C_0(]-1,0[)$ is  $(K\otimes C_0(]-1,0[)) \otimes_{\tilde\kappa_0} C_0(]-1,0[;A_1)$ and $\overline{K_0.I_1}$ is $K_1$.
\end{proof}

\noindent We will also need the following lemmas

\begin{lemma}\label{basicprop}

\begin{enumerate}
\item If $f\in C([-1,1];\R)$, then $f$ is a self-adjoint element in $Z(M(D))$ and more generally for any  $D$-Hilbert module $\mathcal{E}$  then the right multiplication by $f$ induces 
a morphism $\hat f \in Z(\mathcal{L}_D(\mathcal{E}))$ such that the map $f\mapsto \hat f$ is a algebra morphism.
\item Let $f$ in $C_0(]-1,0[;\R)$. Then $f\in I_1\cap Z(D) $ and the right multiplication by $f$ induces a morphism $\hat f$ of $\mathcal{L}_D(K_0,K_1)$ such that $\hat f^* \hat f=\hat{f^2}$ in $\mathcal{L}_D(\K_0)$ and
$\hat f \hat f^*=\hat{f^2}$ in $\mathcal{L}_D(\K_1)$
\item Let $f$ in $C_0(]0,1[;\R)$.  Then $f\in I_2\cap Z(D) $ and the right multiplication by $f$ induces a morphism $\hat f$ of $\mathcal{L}_D(K_0,K_2)$ such that $\hat f^* \hat f=\hat{f^2}$ in $\mathcal{L}_D(\K_0)$ and $\hat f \hat f^*=\hat{f^2}$ in $\mathcal{L}_D(\K_2)$
\end{enumerate}
\end{lemma}
\noindent The first point is pretty obvious and $(2)$ and $(3)$ are also clear  in view of lemma \ref{lemmaK}.

\begin{lemma}\label{compactprop}
\begin{enumerate}
\item If $f\in C_0(]-1,1[;\R)$ then for any $B$-module $\mathcal{E}$ and $F\in \mathcal{K}_B(\mathcal{E})$, we have $(F\otimes 1_S)\otimes_{\kappa_0} 1_D \hat f$ is  a compact operator of 
$(\mathcal{E}\otimes S)\otimes_{\kappa_0} D$
\item If $f\in C_0(]-1,0[;\R)$ then for any $A_1$-module  $\mathcal{E}$ and $F\in \mathcal{K}_{A_1}(\mathcal{E})$, we have $F\otimes 1_{C_0(]-1,0[;\R)}\otimes_{\kappa_1} 1_D \hat f$ is a compact operator of
 $(\mathcal{E}\otimes C_0(]-1,0[))\otimes_{\kappa_1} D$
 \item Similarily for $f\in C_0(]0,1[;\R)$ and $A_2$-modules.
 \end{enumerate}
 \end{lemma}

\begin{proof}
Point $(2)$ and $(3)$ are similar to $(1)$. To do $(1)$, let $F$ be the rank one operator $\theta_{\xi,\eta}$ for $\xi$ and $\eta$ vectors in $\mathcal{E}$ which is defined as
$\theta_{\xi,\eta}(x)= \xi<\eta,x>$ for all $x$ in $\mathcal{E}$. Then $(F\otimes 1_S)\otimes_{\kappa_0} 1_D \hat f$ is $\theta_{\xi\otimes f_2\otimes f_2,\eta\otimes f_2\otimes f_2} \hat f_1$ and therefore compact for any function $f=f_1f_2^4$ with $f_1$ and $f_2$ in $C_0(]-1,1[;\R)$.  As any function can be written like that, use for example the polar decomposition, we get our result.
\end{proof}
\noindent Define now two functions in $C([-1,1];\R)$ : $C^+(t)$ is $\cos(\pi t)$ if $t\geq 0$ and  $1$ if $t\leq 0$, the function $C^-(t)$ is  $\cos(\pi t)$ if $t\leq 0$ and $1$ if $t \geq 0$.
Similarly, we have two functions in $S$ ; $S^+$ is $\sin(\pi t)$ if $t\geq 0$ and  $0$ if $t\leq 0$, the function $S^-(t)$ is  $\sin(\pi t)$ if $t\leq 0$ and $0$ if $t \geq 0$.
And finally $T$ is the identity function of $C([-1,1];\R)$.

\noindent With the notation of the first part, we have a natural $D$-module 
$$H=(H_1\otimes C_0(]-1,0[))\otimes_{\kappa_1} D\oplus (H_2\otimes C_0(]0,1[))\otimes_{\kappa_2} D\oplus (K\otimes S)\otimes_{\kappa_0} D.$$
It is also clear that H is endowed with a natural (left) action of $A_f$ as $H_1, H_2$ and $K$ have it.

\vspace{0.2cm}

\noindent Let $G$ be the operator of $\mathcal{L}_D(H)$ defined in matrix form by 

$$G=\begin{pmatrix} \widehat{C^-} & 0 & - ((F_1\otimes 1_{C_0(]-1,0[)})^*\otimes_{\kappa_1} 1 )\widehat{S^-} \\
0 & -\widehat C^+ & ((F_2\otimes 1_{C_0(]0,1[)})^*\otimes_{\kappa_2} 1 )\widehat{S^+}\\
-  \widehat{S^-}^*((F_1\otimes 1_{C_0(]-1,0[)})\otimes_{\kappa_1} 1)  & \widehat{S^+}^* ((F_2\otimes 1_{C_0(]0,1[)})\otimes_{\kappa_2} 1) & Z\\
\end{pmatrix}$$
where $Z= -\widehat{C^-} (q_1\otimes 1_S)\otimes_{\kappa_0} 1+\widehat{ C^+}(q_2\otimes 1_S)\otimes_{\kappa_0} 1-\widehat{T} (q_0\otimes 1_S)\otimes_{\kappa_0} 1$. Thanks to  lemma \ref{basicprop}, $G$ is well-defined. Moreover the following holds.

\begin{proposition}
The operator $G$ verifies $G^2-1$ is a compact operator of $H$ and $G$ commutes modulo compact operators with the action of $A_f$.
\end{proposition}

\begin{proof}
Computing $G^2$ one gets as upper left $2\times 2$  corner :
$$\begin{pmatrix}
\widehat{C^-}^2+F_1^*\otimes_{\kappa_1} 1 \widehat{S^-}  \widehat{S^-}^*F_1\otimes_{\kappa_1} 1 & F_1^*\otimes_{\kappa_1} 1 \widehat{S^-} \widehat{S^+}^* F_2\otimes_{\kappa_2} 1 \\
F_1^*\otimes_{\kappa_1} 1 \widehat{S^-} \widehat{S^+}^* F_2\otimes_{\kappa_2} 1  & \widehat{C^+}^2+F_2^*\otimes_{\kappa_1} 1 \widehat{S^+}  \widehat{S^+}^*F_2\otimes_{\kappa_1} 1\\
\end{pmatrix}$$
As $F_1^*F_1$ is the identity modulo compact operator, using lemma \ref{compactprop} ( the function $(S^-)^2$ is in $C_0(]-1,1[)$ ) one has that $F_1^*\otimes_{\kappa_1} 1 \widehat{(S^-)^2} F_1\otimes_{\kappa_1} 1$ is $\widehat{(S^-)^2}$ modulo compact operators. Recalling also that $F_1^*F_2=0$, one gets that  this matrix is the identity modulo compact operators.

\noindent Let's focus now on the last row of $G^2$. We get first $-\widehat{C^-} F_1^*\otimes_{\kappa_1} 1 \widehat{S^-} - F_1^*\otimes_{\kappa_1} 1 \widehat{S^-} Z$.  As $F_1^*q_1\otimes_{i_1} 1=F_1^*$ and $F_1^*q_2\otimes_{i_1} 1=0$ along with $F_1^*q_0\otimes_{i_1} 1=0$, $F_1^*\otimes_{\kappa_1} 1 \widehat{S^-} Z$ is $-F_1^*\otimes_{\kappa_1} 1 \widehat{S^-}\widehat{C^-}$.
The second composant of that row is treated in the same way. 
Finally the last composant is $\widehat{S^-}^2(F_1F_1^*)\otimes_{\kappa_1} 1+\widehat{S^+}^2(F_2F_2^*)\otimes_{\kappa_1} 1+ \widehat{C^-}^2 (q_1\otimes 1_S)\otimes_{\kappa_0} 1+\widehat{ C^+}^2(q_2\otimes 1_S)\otimes_{\kappa_0} 1+\widehat{T}^2 (q_0\otimes 1_S)\otimes_{\kappa_0} 1$ as $q_0,q_1,q_2$ are commuting projections.
But $F_lF_l^*$ is $q_l\otimes_{i_l} 1$  so $\widehat{S^-}^2(F_1F_1^*)\otimes_{\kappa_1} 1$ is $\widehat{S^-}^2 (q_1\otimes 1_S)\otimes_{\kappa_0} 1$.
Hence, as $q_1+q_2+q_0=1$, the last component is $1+ \widehat{T^2-1} (q_0\otimes 1_S)\otimes_{\kappa_0} 1$. As $T^2-1$ is in $C_0(]-1,1[)$ and $q_0$ is compact, this composant is then $1$ modulo compact
operator.

\noindent Addressing now the compact commutation with the left action of $A_f$, it is very obvious using \ref{compactprop} for every composant of $G$ except $Z$ as it
 contains multiplication with functions not in $C_0(]-1,1[)$. 
 So let $a$ be in $A_1$. We need to compute $[Z,\rho(a)\otimes_{\kappa_0} 1]$. But we know that $[q_1,\rho(a)]=0$. As $q_2=1-q_1-q_0$  we get that $[Z,\rho(a)\otimes_{\kappa_0} 1]=
 -(\widehat{C^+ + T})[q_0,\rho(a)]\otimes_{\kappa_0}1$ which is compact as $C^+ + T$ is a function that vanishes on $-1$ and $1$. 
 The case when $a$ is in $A_2$ is treated in a similar way, hence the compact commutation property is proved for all $a$ in $A_f$.
 \end{proof}
 
 \noindent As a consequence, the couple $(H, G)$ defines an element of $KK^1(A_f, D)$ which we will call $x$ in the sequel.

 \subsection{K-equivalence}
In all the following proofs we will very often use the external tensor product of Kasparov elements. Instead of the traditional notation $\tau_C(x)$
 for the tensorisation with the algebra $C$ of an element $x$ in $KK^*(A,B)$,  we will write $1_C\otimes x$ for the element in $KK^*(C\otimes A, C\otimes B)$ or $x\otimes 1_C$
 for the element in $KK^*(A\otimes C, B\otimes C)$.  Of course $B\otimes C $ is (non canonically) isomorphic to $C\otimes B$, but as we will perform several times this operation, the order will matter.  Note that we do not specify the tensor norm as the algebra $C$ we will be using is alway nuclear. Also when $\pi$ is a morphism between $A$ and $B$, we will write $[\pi]$ for the canonical element in $KK^0(A,B)$. We will denote by $b$ the element of $KK^1(\mathbb{C}, S)$ which is defined on the $S$ Hilbert module $S$ itself by the operator $\widehat{T}$. It is well known that $b$ is invertible.  

\begin{proposition}\label{subequiv}
 With the hypothesis of \ref{bigthm}, one has in $KK^1(A_f, A_f\otimes S)$ that $x\otimes_D [j]$ is homotopic to $ ( 1_{A_f}\otimes b)\otimes_{A_f\otimes S} ([Id_{A_f}]\otimes 1_S)$
\end{proposition}
 
 \begin{proof}
 To prove that we will choose the representant of $[Id_{A_f}]$ that appear in \ref{sub-K-equivalence} and show that its Kasparov product with $b$ is  homotopic to $x\otimes_D [j]$.
 Call $j_l$ for $l=1,2$ the inclusions of $A_l$ in $A_f$ and $j_0=j_1\circ i_1=j_2\circ i_2$ the inclusion of $B$ in $A_f$. First it is obvious that $H\otimes_{j} (A\otimes S)$ is $H_1\otimes_{j_1}A_f\otimes C_0(]-1,0[)\oplus H_2\otimes_{j_2} A_f\otimes C_0(]0,1[)\oplus K\otimes_{j_0} A_f\otimes S$
 which is not quite the same as $(H_1\otimes_{j_1}A_f\oplus H_2\otimes_{j_2} A_f\otimes\oplus K\otimes_{j_0} A_f)\otimes S$. So we will realize now an homotopy to fix that.
 
 \begin{lemma}
 Consider the following two spaces : $\Delta_1=\{(t,s)\in \R^2,  0\leq s\leq 1, \,-1< t<s\}$ and $\Delta_2=\{(t,s)\in \R^2,  0\leq s\leq 1 ,\,-s< t<1\}$. The Hilbert module $\overline H=
 H_1\otimes_{j_1}A_f\otimes C_0(\Delta_1)\oplus H_2\otimes_{j_2} A_f\otimes C_0(\Delta_2)\oplus K\otimes_{j_0} A_f\otimes S\otimes C([0,1])$ is endowed with a natural structure of
 $A_f\otimes S\otimes C([0,1])$ Hilbert module and $A_f$ left action. Moreover the operator 
 $$\overline G= \begin{pmatrix} \widehat{C^-}\otimes 1_{C([0,1])} & 0 & - F_1^*\otimes_{j_1} 1\otimes 1_{\Delta_1} \widehat{S^-}\otimes 1_{C([0,1])}  \\
0 & -\widehat C^+ \otimes 1_{C([0,1])}& F_2^*\otimes_{j_2} 1\otimes 1_{\Delta_2}  \widehat{S^+}\otimes 1_{C([0,1])}\\
-  \widehat{S^-}^*\otimes 1_{C([0,1])} \,F_1\otimes_{j_1} 1\otimes 1_{\Delta_1}  & \widehat{S^+}^*\otimes 1_{C([0,1])}\, F_2\otimes_{j_2} 1 \otimes_{\Delta_2} & \overline Z\\
\end{pmatrix}
$$
with $\overline Z= \widetilde Z\otimes 1_{C([0,1])} $ where $\widetilde Z=-\widehat{C^-} q_1\otimes_{j_0} 1\otimes 1_S+\widehat{ C^+}q_2\otimes_{j_0} 1\otimes 1_S  -\widehat{T} q_0\otimes_{j_0} 1\otimes 1_S$
 makes the pair $(\overline H, \overline G)$ into an element of $KK^1(A_f, A\otimes S\otimes C([0,1]))$ for which the evaluation at $t=0$  is $x\otimes_{D} [j]$ and the evaluation at $t=1$
 has $(H_1\otimes_{j_1}A_f\oplus H_2\otimes_{j_2} A_f\otimes\oplus K\otimes_{j_0} A_f)\otimes S$ as module and
 $\widetilde G=\begin{pmatrix}
 \widehat{C^-} & 0 & - F_1^*\otimes_{j_1} 1\otimes 1_S \widehat{S^-} \\
0 & -\widehat C^+ & F_2^*\otimes_{j_2} 1\otimes 1_S \widehat{S^+}\\
-  \widehat{S^-}^*F_1\otimes_{j_1} 1\otimes 1_S  & \widehat{S^+}^* F_2\otimes_{j_2} 1 & \widetilde Z\\
 \end{pmatrix}$ as operator.\end{lemma}
 
 \begin{proof} As it is a straightforward check, details will be omitted.
 \end{proof}
  
  \noindent Then one easily checks that $\widetilde G$ is an $\widehat{T}\otimes_{A_f} 1$ connection.
 Indeed as $H_1\otimes_{j_1}A_f\oplus H_2\otimes_{j_2} A_f$ is of grading $0$  and $K\otimes {j_0} A_f$ of grading $-1$, one need to check that, when evaluating on 
 $-1$, $\widetilde G$ does the same thing as $\widehat T$ i.e. is the matrix $\begin{pmatrix}-1&0&0\\ 0&-1&0\\ 0&0&1\\ \end{pmatrix}$ and, when evaluating on  $1$, $G$ is the opposite matrix.
 It is indeed the case as $q_1+q_2+q_0=1$.
 
 \vspace{0.2cm}
 
 \noindent Lastly one need the following lemma where the operator $F$ of \ref{sub-K-equivalence} appears.
 
 \begin{lemma}
 The anti-commutator of $\widetilde G$ and $F\otimes 1_S$ is positive.
 
 \end{lemma}
 
 \begin{proof}
 To do that, we will decompose $\widetilde G$ in its diagonal and anti-diagonal part.
 It is clear that $\begin{pmatrix} \widehat{C^-} & 0 & 0 \\
0 & -\widehat C^+ & 0\\
0 & 0 & \widetilde Z\\
\end{pmatrix}$ and $\begin{pmatrix} 0 & 0 &  F_1^*\otimes_{i_1} 1 \otimes 1_S \\
0 & 0 & F_2^*\otimes_{\i_2} 1 \otimes 1_S\\
  F_1\otimes_{i_1} 1 \otimes 1_S &  F_2\otimes_{i_2} 1 \otimes 1_S & 0\\
\end{pmatrix}$ anti-commutes modulo compact operator as we have (modulo compact operator)  $q_1F_1=F_1$ and $q_2F_1= q_0 F_1=0$. On the other hand the anti-commutator with the anti-diagonal part is
 $$\begin{pmatrix}
 -2 (F_1^*F_1)\otimes_{j_1} 1\otimes 1_S \widehat{S^-} &0&0\\
 0& 2 (F_2^*F_2)\otimes_{j_2} 1\otimes 1_S \widehat{S^+}&0\\
 0&0&-2 q_1\otimes_{j_0} 1\otimes 1_S\widehat{S^-} +2 q_2 \otimes_{j_2} 1\otimes 1_S \widehat{S^+} \\
 \end{pmatrix}$$
 As $-S^-$ and $S^+$ are positive functions and $q_1$ and $q_2$ commutes, the previous matrix is a diagonal matrix of positive operators hence positive.
 \end{proof}
 
 \noindent Using Connes- Skandalis characterization of the Kasparov product, we have established that $\widetilde G$ is  a representant of the  Fredholm operator for the product
 $ ( 1_{A_f}\otimes b)\otimes_{A_f\otimes S} ([Id_{A_f}]\otimes 1_S)$.  Our proposition is henceforth proven.
 \end{proof}
 
 \noindent We need now the following two lemmas to get some information about $[j]\otimes_{A_f\otimes S} (x\otimes 1_S)$ as an element of $KK^1(D,D\otimes S)$.
 
 \begin{lemma}\label{lemmaev0}
 Call $ev_0$ the morphism from $D$ to $B$ that evaluates a function at $0$.
 Then we have in $KK^1(D, B\otimes S)$ that $[j]\otimes_{A_f\otimes S} ( (x\otimes_D [ev_0])\otimes 1_S)= - [ev_0]\otimes_B(1_B\otimes b)$.
  \end{lemma}
 
 \begin{proof}
 Let's first describe the left hand side. The Hilbert module is $K\otimes 1_S$ as the module $(H_1\otimes C_0(]-1,0[))\otimes_{\kappa_1} D\otimes_{ev_0} B$ is $0$.
 The left $D$ action is given by $(\rho\otimes 1_S)\circ j$ and the operator is just $(-q_1+q_2)\otimes 1_S$.
 We can replace this operator with $G_0=(-q_1+q_2)\otimes 1_S-\widehat{T}q_0\otimes 1_S$ as for any $f$ in $D$, $(\rho\otimes 1_S)\circ j(f)\,\widehat{T}q_0\otimes 1_S$ is compact.
 Note now  that the evaluation at $-1$  of $G_0$ is $(1-2q_1)$ and at $-1$ is $2q_2-1$.
 It then enables us to do an homotopy.  Consider the pair $(K\otimes S\otimes C([0,1]), G_0\otimes 1_{C([0,1])})$ where the left action of $D$ is defined now for any $f$ in $D$ and
 $k\in C(]-1,1[\times[0,1];K)$ as $(f.k)(t,s)=\rho(f(t(1-s))) k(t,s)$. This is still a Kasparov element as $(G_0^2-1)\otimes 1_{C([0,1])} =(\widehat{(T^2-1)}q_0\otimes 1_S)\otimes 1_{C([0,1])}$ hence compact.
 Also the commutator of the left action with the operator $G_0\otimes 1$ is compact. Indeed, as $q_0$ is compact, it is only necessary to check that the evaluation at $-1$ or $1$ of any
 commutator is $0$. But this is true as $[q_1,\rho(A_1)]=0$ and $[q_2,\rho(A_2)]=0$. Therefore $[j]\otimes_{A_f\otimes S} ( (x\otimes_D [ev_0])\otimes 1_S)$ is homotopic to an element of $KK^1(D,B\otimes S)$ wich is described with the pair $(K\otimes S, G_0)$ where $D$ acts on $K\otimes S$ as the
 constant morphism $\rho\circ ev_0$. So it is $[ev_0]\otimes_B z$ with $z$ an element of $KK^1(B,B\otimes S)$ which is only non trivial on $q_0K\otimes S\simeq B\otimes S$ where $G_0$ acts as $-\widehat{T}$. 
 Thus $z=- 1_B\otimes b$.  
 \end{proof}
 
\noindent Recall that for $l=1,2$, $\kappa_l$ is the inclusion of $A_l\otimes C(]-1,0[)$ in $D$. To be precise we will use $\bar\kappa_l$  for the induced map from $A_l\otimes S$ to $D$ via the isomorphism of
$C(]-1,0[)$ with $S$.
 
 \begin{lemma}\label{lemma-vertex}
 For all $l=1,2$, one has $[j_l]\otimes_{A_f} x=([Id_{A_l}]\otimes b)\otimes_{A_l\otimes S} [\bar\kappa_l]\in KK^1(A_l,D)$.
  \end{lemma}
 
 \begin{proof}
 We will do the lemma for $l=1$. The element $[j_1]\otimes_{A_f} x$ as the same module and operator that $x$, the only change is that we only consider a left action of $A_1$.
  We  first perform a compact perturbation of the operator $G$. With the operators $\overline F_l$ defined before  \ref{LemmaCompactCommutation}, consider
  
  $$G_1=\begin{pmatrix} \widehat{C^-} & 0 & - F_1^*\otimes_{\kappa_1} 1 \widehat{S^-} \\
0 & -\widehat C^+ & \overline F_2^*\otimes_{\kappa_2} 1 \widehat{S^+}\\
-  \widehat{S^-}^*F_1\otimes_{\kappa_1} 1  & \widehat{S^+}^* \overline F_2\otimes_{\kappa_2} 1 & \overline Z\\
\end{pmatrix}$$
where $\overline Z = -\widehat{C^-} (q_1\otimes 1_S)\otimes_{\kappa_0} 1+\widehat{ C^+}(1- q_1\otimes 1_S)\otimes_{\kappa_0} 1$.

\noindent As $F_2-\overline F_2$ is compact (see \ref{LemmaCompactCommutation} ) and $\overline Z- Z= \widehat{C^+ +T}(q_0\otimes 1_S)\otimes_{\kappa_0} 1$  is compact as
$C^+ +T$ is in $S$, we get the same element of $KK^1(A_1,D)$. Observe now that when evaluating at any positive $t$, $G_1^2$ is the identity because $\overline F_2$ is an isometry and $\widehat{S^-} F_1\otimes_{\kappa_1} 1$ vanishes and that  for any $t$, $G_1$ commutes exactly with the left action of $A_1$ as $F_1$ and $\overline F_2$ does.
 
 \vspace{0.2cm}
 
 \noindent We will now construct an homotopy to remove the $[0,1[$ part of our module.
 Consider the space $\Delta_3=\{(t,s)\in\R\, 0<s\leq 1,\, 0<t<s\}$ and $\Delta_4=\{(t,s)\in\R\, 0\leq s\leq 1,\, -1<t<s\}$ which are open in $]-1,1[\times [0,1]$.
 Hence we also have a natural imbedding $\delta_4$ of $C_0(\Delta_4;B)$ in $D\otimes C([0,1])$ and $\delta_3$ of $C_0(\Delta_3;A_2)$ in $D\otimes C([0,1])$.
 Then $\widetilde H= (H_1\otimes C_0(]-1,0[))\otimes_{\kappa_1} D\otimes C([0,1])  \oplus (H_2\otimes C_0(\Delta_3[))\otimes_{\delta_3} D\otimes C([0,1])  \oplus (K\otimes C_0(\Delta_4)\otimes_{\delta_4} D\otimes C([0,1])$ is well defined and the pair $(\widetilde H, \widetilde G_1)$ is a Kasparov element in $KK^1(A_1,D\otimes C([0,1]))$. 
  Indeed the only thing to check is whether $ \widetilde G_1^2$ is the identity modulo compact operator as $\widetilde G_1$ has exact commutation with the action of $A_1$.
  But this is true by the previous observation.
  \vspace{0.2cm}
  
  \noindent Therefore $[j_l]\otimes_{A_f} x$ can be represented by the evaluation at $0$ of this Kasparov element. Let's describe it: the module part is
  $(H_1\oplus K\otimes_{i_1} A_1 )\otimes C_0(]-1,0[)\otimes_{\kappa_1} D$ with obvious left $A_1$ action as $(K\otimes C_0(]-1,0[))\otimes_{\kappa_0} D$ is isomorphic to
  $(K\otimes_{i_1}A_1)\otimes C_0(]-1,0[)\otimes_{\kappa_1} D$. With this identification,  the operator is
  $$E_1=\begin{pmatrix}
  \widehat{C^-}  & - F_1^*\otimes 1_{C_0(]-1,0[)} \otimes_{\kappa_1} 1 \widehat{S^-} \\
-  \widehat{S^-}^*F_1\otimes 1_{C_0(]-1,0[)}\otimes_{\kappa_1} 1  &   -\widehat{C^-} (q_1\otimes_{i_1} 1\otimes 1_{C_0(]-1,0[})\otimes_{\kappa_1} 1+(1- q_1 \otimes_{i_1} 1 \otimes1_{C_0(]-1,0[})\otimes_{\kappa_1} 1\\
\end{pmatrix}$$.

\noindent It is then clear, after identifying $C_0(]-1,0[)$ with $S$, that  $[j_l]\otimes_{A_f} x$ is $z\otimes_{A_1} [\bar\kappa_1]$ with $z$ in $KK^1(A_1, A_1\otimes S)$. By recalling that $1-q_1$ commutes with the left action of $A_1$, it is obvious that $z$ is represented by the pair $((H_1\oplus q_1K\otimes_{i_1} A_1)\otimes S, \overline E_1)$
with $\overline E_1=\begin{pmatrix}
  \widehat{C_1}  & - F_1^*\otimes 1_{S_1} \widehat{S_1} \\
-  \widehat{S_1}^*F_1\otimes 1_{S}  &   -\widehat{C_1} (q_1\otimes_{i_1} 1\otimes 1_{S})\\
\end{pmatrix}$
where $C_1$ is the function $\cos(\pi(t/2-1/2))$ and $S_1$ the function $\sin(\pi(t/2-1/2))$.
 \vspace{0.2cm}

\noindent Following the proof of \ref{subequiv}, $z$ is obviously the product $z'\otimes b$ where $z'$ is the element of $KK^0(A_1,A_1)$ given by
 the module $H_1\oplus q_1K\otimes_{i_1} A_1$ with $H_1$ positively graded and the obvious left action of $A_1$ and the operator $\begin{pmatrix}
 0&F_1^*\\ F_1&0\end{pmatrix}$. We will be finish when we prove that $z'$ is $[Id_{A_1}]$.  To do this we represent $z'\oplus -[Id_{A_1}]$ by the module $H_1\oplus q_1K\otimes_{i_1} A_1\oplus q_0K\otimes_{i_1} A_1\simeq
 H_1\oplus (1-q_2) K\otimes_{i_1} A_1$ and the previous operator. But it is a compact perturbation of
 $\begin{pmatrix}
 0&\overline{F_1}^*\\ \overline{F_1}&0\end{pmatrix}$.   This last operator is homotopic via a simple rotation to 
 $\begin{pmatrix}
 \overline{F}_1^* \overline{F}_1&0\\ 0&\overline{F}_1\overline{F}_1^* \end{pmatrix}$ hence trivial as $\overline{F}_1^* \overline{F}_1=1$ and $\overline{F}_1\overline{F}_1^*=1$ modulo compact operators as observed
 before \ref{LemmaCompactCommutation}.
 \end{proof}
 
 \vspace{0.2cm}
 \noindent We are now ready to prove our theorem \ref{bigthm}.
 
 \begin{proof}
 Call $a \in KK^1(S, \mathbb{C})$ the inverse of $b$. The element $y=(1_{A_f}\otimes a)\otimes_{A_f} x$ is an element of $KK^0(A\otimes S, D)$. We claim that this is the inverse of $[j]$. Indeed thanks to \ref{subequiv} we have that 
 $$y\otimes_D [j]=(1_{A_f}\otimes a)\otimes_{A_f} x\otimes_D [j]= (1_{A_f}\otimes a)\otimes_{A_f} (1_{A_f}\otimes b)\otimes_{A_f\otimes S} ([Id_{A_f}]\otimes 1_S).$$
  As $a\otimes_{ \mathbb{C}} b=[Id_S]$ we get that $
 y\otimes_D [j]= (1_A\otimes [Id_S])\otimes_{A_f\otimes S}([Id_{A_f}]\otimes 1_S)$ is $[Id_{A_f\otimes S}]$. To prove the reverse equality, we will need a trick that can be found already in \cite{Pi86}.
  Observe first that for any $l=1,2$ and using \ref{lemma-vertex},
  \begin{eqnarray*}
  [\bar\kappa_l]\otimes_D\otimes [j]\otimes_{A_f\otimes S} y&=&[j\circ \bar\kappa_l]\otimes_{A_f\otimes S} y=([j_l]\otimes 1_S)\otimes_{A_f\otimes S} (1_{A_f}\otimes a)\otimes_{A_f}x\\
  &=&(1_{A_l}\otimes a)\otimes_{A_l}[j_l]\otimes_{A_f}x\\
  &=&(1_{A_l}\otimes a)\otimes_{A_l}(1_{A_l}\otimes b)\otimes_{A_l}([Id_{A_l}]\otimes 1_S)\otimes_{A_l\otimes S} [\bar\kappa_l]\\
  &=& [\bar\kappa_l].
  \end{eqnarray*}
  \noindent We need now to compute $[j]\otimes_{A_f\otimes S} y\otimes_D [ev_0]$. To do this we will use the following lemma.
\begin{lemma}\label{lemma-edge}
In $KK^1(D\otimes S, A\otimes S)$, one has
$([j]\otimes_{A_f\otimes S}(1_{A_f}\otimes a))\otimes 1_S= -(1_D\otimes a)\otimes_D [j].$
\end{lemma}

\begin{proof}
Indeed, $$(1_D\otimes b)\otimes_{D\otimes S}([j]\otimes_{A_f\otimes S}(1_{A_f}\otimes a))\otimes 1_S=[j]\otimes_{A_f\otimes S}(1_{A_f}\otimes(1_S\otimes b)\otimes_{S\otimes S}
(a\otimes 1_S)).$$
If $\Sigma$ is the flip automorphism of $S\otimes S$ then clearly $[\Sigma]=-[Id_{S\otimes S}]$ in $KK^0(S\otimes S, S\otimes S)$. 
As a consequence $(1_S\otimes b)\otimes_{S\otimes S}(a\otimes 1_S)=-1_S\otimes ( b\otimes_{\mathbb{C}} a)=-[Id_S]$. 
Hence $$(1_D\otimes b)\otimes_{D\otimes S}([j]\otimes_{A_f\otimes S}(1_{A_f}\otimes a))\otimes 1_S)=-[j].$$ Multiplying both side by $1_D\otimes a$ gives the result.
\end{proof}
\vspace{0.2cm}

\noindent In view of the lemma and \ref{lemmaev0} one has:
\begin{eqnarray*}
  ([j]\otimes_{A_f\otimes S} y\otimes_D [ev_0])\otimes 1_S&=&-(1_D\otimes a)\otimes_D ([j]\otimes_{A_f\otimes S} (x\otimes_D [ev_0])\otimes 1_S)\\
  &=&+ (1_D\otimes a)\otimes_D [ev_0] \otimes_B (1_B\otimes b)\\
  &=&(1_D\otimes a)\otimes_D (1_D\otimes b)\otimes_{D\otimes S}( [ev_0]\otimes 1_S)\\
  &=&[ev_0]\otimes  1_S
  \end{eqnarray*}
\noindent As $-\otimes 1_S$ from $KK(B_1,B_2)$ to $KK(B_1\otimes S, B_2\otimes S)$ is an isomorphism for any $B_1$ and $B_2$, we get
$[j]\otimes_{A_f\otimes S} y\otimes_D [ev_0]=[ev_0]$. Denote now $q= [Id_D]-[j]\otimes_{A_f\otimes S} y$. As $y\otimes_D [j]=[Id_{A_f\otimes S}]$, $q$ is an idempotent in the ring $KK^0(D, D)$.  On the other end, $D$ fits into a short exact sequence 
$$0\rightarrow A_1\otimes S\oplus A_2\otimes S\overset{\bar\kappa_1\oplus \bar\kappa_2}{\longrightarrow} D\overset{ev_0}\longrightarrow B\rightarrow 0.$$
The induced six terms exact sequence for the functor $KK^0(D,-)$ then shows that, as $q\otimes_D [ev_0]=0$, there exist $q_l$ in
$KK^0( D,A_l)$ for $l=1,2$ such that $q=(q_1\oplus q_2)\otimes_{A_1\oplus A_2}([\bar\kappa_1]\oplus [\bar\kappa_2] )$. So $q=q\otimes_D q =(q_1\oplus q_2)\otimes_{A_1\oplus A_2}([\bar\kappa_1]\oplus [\bar\kappa_2] )\otimes_D q=0$.
 because $[\bar\kappa_l]\otimes_D q=0$ for $l=1,2$ as observed before \ref{lemma-edge}.
 Therefore $[Id_D]= [j]\otimes_{A_f\otimes S} y$ and the K-equivalence between $A_f$ and $D$ is established.
 \end{proof}
 
 \noindent We obtain the following immediate corollary.
  \begin{corollary}
 The amalgamated free product of two discrete discrete quantum groups is $K$-amenable if and only if the two initial discrete quantum groups are $K$-amenable.
 \end{corollary}
 
 \begin{proof}
One way is obvious.  Let us prove the converse. Let $G_1, G_2, H$ be compact quantum groups and suppose that $\widehat{H}$ is a common discrete quantum subgroup of both $\widehat{G_1},\widehat{G_2}$ and $\widehat{G_k}$ is $K$-amenable for $k=1,2$. Write, for $k=1,2$, $C_m(G_k), C_m(H)$ the full C*-algebras and $C(G_k), C(H)$ the reduced C*-algebra and view $C_m(H)\subset C_m(G_k)$, $C(H)\subset C(G_k)$, for $k=1,2$. Let $\widehat{G}$ be the amalgamated free product discrete quantum group. One has $C_m(G)=C_m(G_1)\underset{C_m(H)}{*}C_m(G_2)$ and $C(G)=C(G_1)\underset{C(H)}{\overset{e}{*}}C(G_2)$, where the edge-reduced amalgamated free product is done with respect to the \textit{faithful} Haar states on $C(G_k)$, for $k=1,2$. Let $\lambda_{G_k}\,:\,C_m(G_k)\rightarrow C(G_k)$ be the canonical surjection. By assumption, $\lambda_{G_k}$ is $K$-invertible  for $k=1,2$. Observe that the canonical surjection $\lambda_G\,:\,C_m(G)\rightarrow C(G)$ is given by $\lambda_G=\pi\circ\lambda$, where $\lambda\,:\, C_m(G_1)\underset{C_m(H)}{*}C_m(G_2)\rightarrow C(G_1)\underset{C(H)}{*}C(G_2)$ is the free product of the maps $\lambda_{G_1}$ and $\lambda_{G_2}$ and $\pi\,:\,C(G_1)\underset{C(H)}{*}C(G_2)\rightarrow C(G_1)\underset{C(H)}{\overset{e}{*}}C(G_2)$ is the canonical quotient map. By theorem \ref{TheoremKequivalence} $\pi$ is $K$-invertible and using the exact sequence of the full free product and the five's lemma, $\lambda$ is $K$-invertible.
 \end{proof}

\noindent
{\small\sc Emmanuel Germain} \\
  {\small LMNO, CNRS UMR 6139, Universit\'e de Caen\\
\em E-mail address: \tt emmanuel.germain@unicaen.fr}

\bigskip

\noindent
{\small \sc Pierre FIMA} \\ \nopagebreak
  {\small Univ Paris Diderot, Sorbonne Paris Cit\'e, IMJ-PRG, UMR 7586, F-75013, Paris, France \\
  Sorbonne Universit\'es, UPMC Paris 06, UMR 7586, IMJ-PRG, F-75005, Paris, France \\
  CNRS, UMR 7586, IMJ-PRG, F-75005, Paris, France \\
\em E-mail address: \tt pierre.fima@imj-prg.fr}

\end{document}